\documentclass[10pt]{article}
\usepackage{amssymb,url}

\usepackage{amsmath}
\usepackage{longtable}
\usepackage{tikz}
\usepackage{amssymb}
\usepackage{amsfonts}

\newtheorem{theorem}{Theorem}[section]
\newtheorem{cor}[theorem]{Corollary}
\newtheorem{lemma}[theorem]{Lemma}
\newtheorem{prop}[theorem]{Proposition}

\newtheorem{problem}{Problem}

\newtheorem{unnumber}{}

\newtheorem{prepf}[unnumber]{Proof}
\newenvironment{proof}{\prepf\rm}{\endprepf}

\newcommand{\qed}{\qquad$\Box$}

\newcommand{\Soc}{\mathop{\mathrm{Soc}}\nolimits}

\newcommand{\PSL}{\mathop{\mathrm{PSL}}\nolimits}

\newcommand{\Sz}{\mathop{\mathrm{Sz}}\nolimits}
\newcommand{\GF}{\mathop{\mathrm{GF}}\nolimits}

\newcommand{\Aut}{\mathop{\mathrm{Aut}}\nolimits}
\newcommand{\Out}{\mathop{\mathrm{Out}}\nolimits}

\begin{document}

\title{Criterion of unrecognizability of a finite group by its Gruenberg--Kegel graph}
\author{Peter J. Cameron\footnote{University of St Andrews, St Andrews, UK, email:\texttt{pjc20@st-andrews.ac.uk}} \ and   Natalia~V.~Maslova \footnote{Krasovskii Institute of Mathematics and Mechanics UB RAS, Ekaterinburg, Russia and Ural Federal University, Ekaterinburg, Russia, email: \texttt{butterson@mail.ru}}$\phantom{x}$\footnote{Corresponding author.}}
\date{}

\maketitle

\centerline{Dedicated to the memory of Jan Saxl}

\begin{abstract}
The Gruenberg--Kegel graph $\Gamma(G)$ associated with a finite group $G$ is an undirected graph without loops and multiple edges whose vertices are the prime divisors of $|G|$ and in which vertices $p$ and $q$ are adjacent in $\Gamma(G)$  if and only if $G$ contains an element of order $pq$. This graph has been the subject of much recent interest; one of our goals here is to give a survey of some of this material, relating to groups with the same Gruenberg--Kegel graph. However, our main aim is to prove several new results. Among them are the following.
\begin{itemize}
\item There are infinitely many finite groups with the same Gruenberg--Kegel graph as the Gruenberg--Kegel of a finite group $G$ if and only if there is a finite group $H$ with non-trivial solvable radical such that $\Gamma(G)=\Gamma(H)$.
\item There is a function $F$ on the natural numbers with the property that if a finite $n$-vertex graph whose vertices are labelled by pairwise distinct primes is the Gruenberg--Kegel graph of more than $F(n)$ finite groups, then it is the Gruenberg--Kegel graph of infinitely many finite groups. (The function we give satisfies $F(n)=O(n^7)$, but this is not best possible.)
\item If a finite graph $\Gamma$ whose vertices are labelled by pairwise distinct primes is the Gruenberg--Kegel graph of only finitely many finite groups, then all such groups are almost simple; moreover, $\Gamma$ has at least three pairwise non-adjacent vertices, and each vertex is non-adjacent to at least one other vertex, in particular, $2$ is non-adjacent to at least one odd vertex.
\item Groups whose power graphs, or commuting graphs, are isomorphic have the same Gruenberg--Kegel graph.
\item The groups ${^2}G_2(27)$ and $E_8(2)$ are uniquely determined by the isomorphism types of their Gruenberg--Kegel graphs.
\end{itemize}
In addition, we consider groups whose Gruenberg--Kegel graph has no edges. These are the groups in which every element has prime power order, and have been studied under the name \emph{EPPO groups}; completing this line of research, we give a complete list of such groups.
\end{abstract}

\section{Introduction}

Throughout the paper we consider only finite groups and simple graphs, and henceforth the term group means finite group, the term graph means simple graph (undirected graph without loops and multiple edges).

Let~$G$ be a group. Denote by $\pi(G)$ the set of all prime divisors of the
order of $G$ and by $\omega (G)$ the {\it spectrum} of~$G$, that is, the set of all its element orders. The
set $\omega(G)$ defines the {\it Gruenberg--Kegel graph} (or the {\it prime graph})
$\Gamma(G)$ of~$G$; in this graph the vertex set is $\pi(G)$, and distinct vertices~$p$ and~$q$ are adjacent if and only if $pq\in\omega (G)$.

The concept of Gruenberg--Kegel graph appeared in the unpublished manuscript \cite{Gruen_Keg} by K.~Gruenberg and O.~Kegel, where they have characterized groups with disconnected Gruenberg--Kegel graph. This result was published later in the paper~\cite{Williams} by J.~Williams, who was a student of K.~Gruenberg, and now this theorem is well-known as the Gruenberg--Kegel Theorem (see Lemma~\ref{Gruenberg--Kegel theorem} in section~\ref{Prelim}). The concept of Gruenberg--Kegel graph proved to be very useful  with connection to research of some cohomological questions in integral group rings: the augmentation ideal of an integral group ring is decomposable as a module if and only if the Gruenberg--Kegel graph of the group is disconnected (see \cite{Gruen_Rogg}).

Later connected components of Gruenberg--Kegel graphs of simple groups were described. J.~Williams~\cite{Williams} has obtained this description for all simple groups except simple groups of Lie type in characteristic $2$. Connected components of Gruenberg--Kegel graphs of simple groups of Lie type in characteristic $2$ were described by A.~S.~Kondrat'ev~\cite{Kondr1}, later this result was obtained independently by N.~Iiyori and H.~Yamaki~\cite{Iiyori_Yamaki_1,Iiyori_Yamaki_2}. Unfortunately, all the papers \cite{Williams,Kondr1,Iiyori_Yamaki_1,Iiyori_Yamaki_2} contain rather serious inaccuracies. Most of these inaccuracies was corrected in~\cite{Kondr4}, and then the corrections were finished in~\cite{Kondr2}. Now the correct description of connected components of Gruenberg--Kegel graphs of simple groups can be found, for example, in~\cite{ak} or in~\cite{Mazurov2004}; finite simple groups~$P$ with at least $4$ connected components of Gruenberg--Kegel graph can be found in section~\ref{Prelim} of this paper, see Table~\ref{table}. Criteria of adjacency of vertices in Gruenberg--Kegel graphs of simple groups were obtained by A.~V.~Vasil'ev and E.~P.~Vdovin in~\cite{VasilVdov_2005} with some corrections in~\cite{VasilVdov_2011}. Moreover, at this moment all the cases of coincidence of Gruenberg--Kegel graphs of a simple group and its proper subgroup are described by the second author~\cite{Maslova} and, independently, by T.~Burness and E.~Covato~\cite{Burness_Covato}. In section~\ref{Prelim}, we provide some known properties of Gruenberg--Kegel graphs of groups as well as some facts from group theory and number theory which we use to prove the main results of this paper.

\medskip

This paper aims to discuss the questions:
\begin{itemize}
\item Which groups $G$ are uniquely determined by their Gruenberg--Kegel graph?
\item For which groups are there only finitely many groups with the same  Gruenberg--Kegel graph as $G$?
\item Which groups $G$ are uniquely determined by isomorphism type of their Gruenberg--Kegel graph?
\end{itemize}

The concept of Gruenberg--Kegel graph and the Gruenberg--Kegel Theorem proved very useful for recognition questions of a group by its spectrum. It is easy to see that for groups $G$ and $H$, if $G \cong H$, then $\omega(G)=\omega(H)$; and if $\omega(G)=\omega(H)$, then $\Gamma(G)=\Gamma(H)$; if $\Gamma(G)=\Gamma(H)$, then $\pi(G)=\pi(H)$ and $\Gamma(G)$ and $\Gamma(H)$ are isomorphic as abstract graphs. The converse does not hold in each case, as the following series of examples demonstrates (see, for example, \cite{Atlas}):

\begin{itemize}

\item $S_5\not \cong S_6$ but $\omega(S_5)=\omega(S_6)${\rm;}

\item $\omega(A_5)\not =\omega(A_6)$ but $\Gamma(A_5)=\Gamma(A_6)${\rm;}

\item $\Gamma(A_{10})\not =\Gamma(\Aut(J_2))$ but  $\Gamma(A_{10})$ and $\Gamma(\Aut(J_2))$ are isomorphic as abstract graphs  and $\pi(A_{10}) =\pi(\Aut(J_2))$, see the picture below

\bigskip

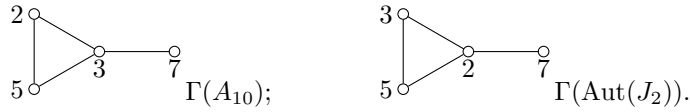
\centerline{
    \begin{tikzpicture}
    \tikzstyle{every node}=[draw,circle,fill=white,minimum size=4pt,
                            inner sep=0pt]
    \draw (0,0) node (2) [label=left:$2$]{}
        -- ++ (-30:1cm) node (3) [label=below:$3$]{}
        -- ++ (0:1cm) node (7) [label=below:$7$]{}
          (-90:1.0cm) node (5) [label=left:$5$]{};
    \draw (5) -- (3);
    \draw (5) -- (2);
    \end{tikzpicture}{$\Gamma(A_{10})$;} \\
    \begin{tikzpicture}
    \tikzstyle{every node}=[draw,circle,fill=white,minimum size=4pt,
                            inner sep=0pt]
    \draw (0,0) node (2) [label=left:$3$]{}
        -- ++ (-30:1cm) node (3) [label=below:$2$]{}
        -- ++ (0:1cm) node (7) [label=below:$7$]{}
          (-90:1.0cm) node (5) [label=left:$5$]{};
    \draw (5) -- (3);
    \draw (5) -- (2);
    \end{tikzpicture}{$\Gamma(\Aut(J_2))$.}
}

\bigskip

\end{itemize}

\noindent We say that the group $G$ is
\begin{itemize}
\item{\it recognizable} by its spectrum (Gruenberg--Kegel graph, respectively) if for each group~$H$, $\omega(G)=\omega(H)$ ($\Gamma(G)=\Gamma(H)$, respectively) if and only if $G \cong H$\/;
\item  {\it $k$-recognizable} by spectrum (Gruenberg--Kegel graph, respectively), where $k$ is a non-negative natural number, if there are exactly $k$ pairwise non-isomorphic groups with the same spectrum (Gruenberg--Kegel graph, respectively) as $G${\rm;}
\item  {\it almost recognizable} by spectrum (Gruenberg--Kegel graph, respectively) if it is $k$-recognizable by spectrum (Gruenberg--Kegel graph, respectively) for some non-negative natural number $k${\rm;}
\item {\it unrecognizable} by spectrum (Gruenberg--Kegel graph, respectively), if there are infinitely many pairwise non-isomorphic groups with the same spectrum (Gruenberg--Kegel graph, respectively) as $G$.
\end{itemize}

If a group $G$ contains a non-trivial solvable normal subgroup, then there are infinitely many groups with the same spectrum as $G$. This proposition, formulated first by W.~Shi, is well-known, and its proof was published in a number of papers (see, for example, \cite{Mazurov_Shi}). Moreover, in~\cite{Mazurov_Shi} the following criteria of unrecognizability of a group by spectrum was obtained.

\begin{theorem}[{\rm see \cite{Mazurov_Shi}}]\label{UnrecSpec} Let $G$ be a group. The following statements are equivalent\/{\rm:}
\begin{itemize}
\item[$(1)$] there exist infinitely many groups $H$ such that $\omega(G)=\omega(H)$\/{\rm;}
\item[$(2)$] there exists a group $H$ with non-trivial solvable radical such that $\omega(G)=\omega(H)$.
\end{itemize}
\end{theorem}

Thus, the question of recognition of a group by its spectrum is interesting only for groups with trivial solvable radical. This question is being actively investigated. We do not pretend to provide a complete survey of this research area but will mention some results which are interesting from our point of view. A remarkable result is that if $G$ is a simple group and $H$ is a group, then $|H|=|G|$ and $\omega(H)=\omega(G)$ if and only if $H \cong G$. (This was put forward as a conjecture in the paper \cite{Shi_Conjecture} and the final step of the proof was made in the paper~\cite{GrechMazVas}.) Moreover, for every nonabelian simple group $G$, apart from a finite number of sporadic, alternating and exceptional groups and apart from several series of classical groups of small dimensions, if  $H$ is a group such that $\omega(H)=\omega(G)$, then $H$ is an almost simple group with socle isomorphic to $G$, therefore {\it almost all finite simple groups are almost recognizable by spectrum}. This result was obtained in a large number of papers and is still in progress in sense of investigation of recognition by spectrum of low-dimensional classical groups. We recommend surveys of the results in this area in papers \cite{Grechkoseeva,Grech_Vas}; moreover, we recommend the following paper by A.~V.~Vasil'ev~\cite{Vasil_2015}, where an important contribution to the solution of the problem of recognition by spectrum for simple classical groups of Lie type was made, and some recent papers as \cite{Staroletov1,Grech_Skres,GrechVasZvezd,Grech_Zvezd}. Moreover, some almost simple groups are recognizable by spectrum, for example, see~\cite{Gorsh_Grish}.

\smallskip

It is easy to see that if a group is recognizable by its Gruenberg--Kegel graph, then it is recognizable by its spectrum. The converse does not hold since, for example, $\Gamma(A_5)=\Gamma(A_6)$, but the group $A_5$ is recognizable by spectrum (see~\cite{Shi}) while the group $A_6$ is not. Some information about groups with the same spectrum as $A_6$ can be found in \cite[Theorem~2]{Lytkin}.

\smallskip

In section~\ref{Rec=>AS}, we prove the following criterion of unrecognizability of a group by its Gruenberg--Kegel graph.

\begin{theorem}\label{UnrecGK} Let $G$ be a group. The following statements are equivalent{\rm:}
\begin{itemize}
\item[$(1)$] there exist infinitely many groups $H$ such that $\Gamma(G)=\Gamma(H)${\rm;}
\item[$(2)$] there exists a group $H$ with non-trivial solvable radical such that $\Gamma(G)=\Gamma(H)$.
\end{itemize}
\end{theorem}

Moreover, in section~\ref{Rec=>AS}, we characterize finite groups which are almost recognizable by Gruenberg--Kegel graph. We prove the following theorem.

\begin{theorem}\label{CoGrAlmSimple} Let $G$ be a group such that $G$ is $k$-recognizable by Gruenberg--Kegel graph for some non-negative integer $k$. Then the following conditions hold{\rm:}
\begin{itemize}
\item[$(1)$] $G$ is almost simple{\rm;}
\item[$(2)$] each group $H$ with $\Gamma(H)=\Gamma(G)$ is almost simple{\rm;}
\item[$(3)$] each vertex of $\Gamma(G)$ is non-adjacent to at least one other vertex, in particular, $2$ is non-adjacent to at least one odd prime in $\Gamma(G)${\rm;}
\item[$(4)$] $\Gamma(G)$ contains at least $3$ pairwise non-adjacent vertices.
\end{itemize}
\end{theorem}

Note that a group with non-simple socle can be recognizable by spectrum, therefore Theorem~\ref{CoGrAlmSimple} can not be generalized for recognition by spectrum. Up to a recent moment there were only two examples of groups with non-simple socle which are recognizable by spectrum, namely, $\Sz(2^7)\times\Sz(2^7)$ (see~\cite{Mazurov_1997}) and $J_4\times J_4$ (see~\cite{Gorsh_Mas}). Recently, I.~B.~Gorshkov has proved that if $m>5$, then the group $\PSL_{2^m}(2)\times \PSL_{2^m}(2)\times \PSL_{2^m}(2)$ is recognizable by spectrum, a preprint of this paper is available on the arXiv (see~\cite{Gorshkov}).

\medskip

We conclude from Theorem~\ref{CoGrAlmSimple} that if a group is recognizable by its Gruenberg--Kegel graph, then the group is almost simple. The following problem naturally arises.
\medskip

\begin{problem} Let $G$ be an almost simple group. Decide whether $G$ is
recognizable, $k$-recognizable for some integer $k>1$, or unrecognizable by its Gruenberg--Kegel graph.
\end{problem}

There are some known results on recognition of a group by its Gruenberg--Kegel graph. Here we again will mention some results which are interesting from our point of view; we do not pretend to provide a complete survey of this research area.

The first result on recognition of a group by its Gruenberg--Kegel graph was obtained by G.~Chen~\cite{Chen}, where it was proved that if $S$ is a sporadic simple group and $H$ is a group with $|H|=|S|$ and $\Gamma(H)=\Gamma(S)$, then $H \cong S$. Later M.~Hagie~\cite{Hagie_2003} described groups with the same Gruenberg--Kegel graphs as simple sporadic groups. In particular, Hagie has proved that the groups $J_1$, $M_{22}$, $M_{23}$, $M_{24}$, and $Co_2$ are recognizable by their Gruenberg--Kegel graphs, the group $M_{11}$ is $2$-recognizable, and the groups $M_{12}$ and $J_2$ are unrecognizable by their Gruenberg--Kegel graphs; moreover, if $S$ is one of groups $O'N$, $Ly$, $Fi_{23}$, $Fi_{24}$, $M$, $BM$, $Th$, $Ru$, and $Co_1$, then $S$ is {\it quasirecognizable} by its Gruenberg--Kegel graph; that is, any group $H$ with $\Gamma(H)=\Gamma(S)$ has a unique nonabelian composition factor which is isomorphic to the group $S$. In 2006, A.~V.~Zavarnitsine~\cite{Zavarnitsine_2006} proved that the group $J_4$ is recognizable by its Gruenberg--Kegel graph, moreover, $J_4$ is the unique group whose Gruenberg--Kegel graph has exactly $6$ connected components. Later, based on Hagie's results, A.~S.~Kondrat'ev~\cite{KondrRecSpor1} has proved that the group $Ru$ is recognizable by its Gruenberg--Kegel graph, the group $HN$ is $2$-recognizable, the group $Fi_{22}$ is $3$-recognizable, the groups $He$, $McL$, and $Co_3$ are unrecognizable by their Gruenberg--Kegel graphs. Recently A.~S.~Kondrat'ev~\cite{KondrRecSpor2} has proved that the groups $J_3$, $Suz$, $O'N$, $Ly$, $Th$, $Fi_{23}$, and $Fi_{24}$ are recognizable by their Gruenberg--Kegel graphs, and the group $HS$ is $2$-recognizable. Thus, at the time of writing, only three large sporadic groups are left for which recognition by the Gruenberg–Kegel graph is not completely settled: $Co_1$, $B$, and $M$. Due to Hagie's result mentioned above, these groups were known to be  quasirecognizable by Gruenberg--Kegel graph. After this paper was submitted to the journal, M.~Lee and T.~Popiel have proved that finite simple sporadic groups $Co_1$, $B$, and $M$ are recognizable by Gruenberg-Kegel graph\footnote{M.~Lee, T.~Popiel, $M$, $B$ and $Co_1$ are recognisable by their prime graphs,  arXiv:2107.12755v1 [math.GR].}.

M.~Hagie~\cite{Hagie_2003} has proved that the group $\PSL_2(11)$ is $2$-recognizable from its Gruenberg--Kegel graph. In \cite{3Khosravi_2007} Bahman Khosravi, Behman Khosravi, and Behrooz Khosravi have proved that  if $p>7$ is a Mersenne prime or a Fermat prime, then the group $\PSL_2(p)$ is recognizable by its Gruenberg--Kegel graph. In \cite{3Khosravi_2007_2} the same authors proved that  if $p > 11$ is a prime number and $p \not \equiv 1 \pmod{12}$, then the group $\PSL_2(p)$ is recognizable by its Gruenberg--Kegel graph. A.~Khosravi and B.~Khosravi~\cite{ABKhosravi} have proved that if $p$ is a prime not in $\{2, 3, 7\}$, then the group $\PSL_2(p^2)$ is $2$-recognizable by its Gruenberg--Kegel graph. In 2008, B.~Khosravi \cite{BKhosravi} proved that if $q = p^k$, where $k > 1$ is odd and $p$ is an odd prime number, then the group $\PSL_2(q)$ is recognizable by its Gruenberg--Kegel graph. Moreover, Z.~Akhlaghi, B.~Khosravi, and M.~Khatami~\cite{Akhlaghi_Khosravi_Khatami_2010}
proved that if $p$ is an odd prime and $k > 1$ is odd, then the group $PGL_2(p^k)$ is recognizable by its Gruenberg--Kegel graph,  A.~Mahmoudifar~\cite{Mahmoudifar} has proved that the group $PGL_2(25)$ is recognizable by its Gruenberg--Kegel graph.

The question of recognition by Gruenberg--Kegel graph of alternating and symmetric groups was studied in~\cite{Khosravi_Moghanjoghi,Gorsh_Staroletov,Staroletov2}.

A.~V.~Zavarnitsine~\cite{Zavarnitsine_2006} has proved that the group $\PSL_3(7)$ is $2$-recognizable by Gruenberg--Kegel graph. Later the question of recognition by Gruenberg--Kegel graph for simple groups $S$ such that $|\pi(S)|\in \{3,4\}$ was studied in~\cite{KondrKhr1,KondrKhr2}.

A remarkable result is that the group $\PSL_{16}(2)$ is recognizable by its Gruenberg--Kegel graph. This result was obtained by B.~Khosravi, B.~Khosravi, and B.~Khosravi~\cite{3Khosravi}, however, in this paper there was a ﬂaw in the proof of Lemma~3.4. A complete proof of the result was obtained later by A.~V.~Zavarnitsine~\cite{Zavarnitsine_2010}. The group $\PSL_{16}(2)$ was the first known example of a group with connected Gruenberg--Kegel graph which is recognizable by its Gruenberg--Kegel graph.

Z.~Momen and B.~Khosravi \cite{Momen_Khosravi} have proved that groups $B_p(3)$ and $C_p(3)$, where $p>3$ is an odd prime, are $2$-recognizable by Gruenberg--Kegel graph.  Later M.~F.~Ghasemabadi, A.~Iranmanesh, and N.~Ahanjideh~\cite{Ghasemabadi_Iranmanesh_Ahanjideh} proved that if $n> 5$ is an odd number, then groups $B_n(3)$ and $C_n(3)$ are $2$-recognizable by Gruenberg--Kegel graph.
A.~Babai and B.~Khosravi~\cite{Babai_Khosravi_2011} have proved that the group ${^2}D_{2^m+1}(3)$ is recognizable  by Gruenberg--Kegel graph.
Later in~\cite{Ghasemabadi_Iranmanesh_Ahanjideh_2012} M.~F.~Ghasemabadi, A.~Iranmanesh, and N.~Ahanjideh proved that if $n \ge 5$ is odd, then the group ${^2}D_n(3)$ is recognizable by Gruenberg--Kegel graph. M.~F.~Ghasemabadi and N.~Ahanjideh~\cite{Ghasemabadi_Ahanjideh} have proved that if $n \ge 6$ is even, then the group $D_n(3)$ is recognizable by its Gruenberg--Kegel graph.
Z.~Akhlaghi, M.~Khatami, and B.~Khosravi \cite{Akhlaghi_Khatami_Khosravi} have proved that if $p$ is an odd prime, then the group $D_p(5)$ is recognizable by its Gruenberg--Kegel graph, and the group $D_p(2)$ is quasirecognizable. Later A.~Babai and B.~Khosravi~\cite{Babai_Khosravi} proved that if $n$ is odd, then the group $D_n(5)$ is recognizable by its Gruenberg--Kegel graph and if $n$ is even, then $D_n(5)$ is quasirecognizable.

A.~V.~Zavarnitsine~\cite{Zavarnitsine_2006} has proved that the groups $G_2(7)$ and ${^2}G_2(q)$ for each $q$ are recognizable by Gruenberg--Kegel graph. A.~S.~Kondrat'ev~\cite{Kondrat'evE7(2)andE7(3),Kondrat'ev2E6(2)} has proved that the groups $E_7(2)$, $E_7(3)$, and ${^2}E_6(2)$ are recognizable by Gruenberg--Kegel graph. W.~Guo, A.~S.~Kondrat'ev, and the second author~\cite{GuoKondrat'evMaslova} proved that the group $E_6(2)$ is recognizable by Gruenberg--Kegel graph. Recognizability of of groups $E_6(3)$ and $^2E_6(3)$ by Gruenberg--Kegel graph has been recently proved by A.~P.~Khramova, the second author, V.~V.~Panshin, and A.~M.~Staroletov\footnote{A.~P.~Khramova, N.~V.~Maslova, V.~V.~Panshin, and A.~M.~Staroletov, Recognition of groups $E_6(3)$ and ${^2}E_6(3)$ by Gruenberg--Kegel graph, in preparation.} in frame of realization of a project ''Gruenberg--Kegel graphs of finite groups'' of The Great Mathematical Workshop organized by Mathematical Center in Akedemgorodok (Novosibirsk, Russia) on July 12--17 and August 16--21, 2021 with an intermodule work in between.

One more remarkable result was obtained in 2013 by A.~V.~Zavarnitsine~\cite{Zavarnitsine_2013}, it was proved that if $G$ is a finite group whose Gruenberg--Kegel graph has exactly $5$ connected components, then $G \cong E_8(q)$, where $q \equiv 0,1,4\pmod{5}$. In particular, groups $E_8(q)$, where $q \equiv 0,1,4\pmod{5}$, are almost recognizable by Gruenberg--Kegel graph.

There are some other results on recognition of a simple group by its Gruenberg--Kegel graph, in particular, some groups of Lie type are known to be quasirecognizable by their Gruenberg--Kegel graphs. For example, we recommend to see papers~\cite{Akhlaghi_Khatami_Khosravi_2}, \cite{Amiri_Asboei_Iranmanesh_Tehranian_1}, \cite{Amiri_Asboei_Iranmanesh_Tehranian_2}, \cite{Babai_Khosravi_2012}, \cite{Babai_Khosravi_2014}, \cite{Babai_Khosravi_2015}, \cite{Beynekalae_Iranmanesh_Ghasemabadi} \cite{Ghasemabadi_Iranmanesh}, \cite{Ghasemabadi_Iranmanesh_2}, \cite{BKhosravi_2}, \cite{Khosravi_Akhlaghi_Khatami}, \cite{Khosravi_Babai_2011}, \cite{Khosravi_Khosravi_Oskouei}, \cite{Khosravi_Moradi_1}, \cite{Khosravi_Moradi_2}, \cite{Khosravi_Moradi_3}, \cite{Mahmoudifar_Khosravi}, \cite{Momen_Khosravi_2}, \cite{Momen_Khosravi_3}, \cite{Moradi_Darafsheh_Iranmanesh}, \cite{Mosavi_Ahanjideh}, \cite{Nosratpour_Darafsheh}, \cite{Zhang_Shi_Shen},
and other papers by Behrooz Khosravi et~al., Anatoly Kondrat'ev et~al, Neda Ahanjideh et~al., W.~Shi et~al, and so on.

\medskip
It is known that if $q$ is odd and $n\ge 3$, then $\Gamma(B_n(q))=\Gamma(C_n(q))$ and $|B_n(q)|=|C_n(q)|$ but these groups are not isomorphic.  Thus, it is natural to consider the following problem.

\begin{problem}\label{123} For which simple groups $S$ is the following true:
ift $G$ is a group with $\Gamma (G)$ = $\Gamma (S)$ and $|G| = |S|$, then $G$
is isomorphic to $S$\/{\rm?}
\end{problem}

Problem~\ref{123} was formulated by Behrooz Khosravi in his survey paper~\cite[Question~4.2]{BKhosravi_survey}, by A.S. Kondrat'ev in frame of the open problems session of the 13th School–Conference on Group Theory Dedicated to V. A. Belonogov’s 85th Birthday (see~\cite[Question~4]{Maslova_Conference}), and was independently formulated by W.~Shi in a partial communication with the second author.

\medskip

Let $\Gamma$ be a simple graph whose vertices are labeled by pairwise distinct primes.
We call $\Gamma$ a {\it labeled graph}.
Note that there are examples of labeled graphs which are not equal (and not even isomorphic as abstract graphs) to Gruenberg--Kegel graphs of groups. For example, by the Gruenberg-Kegel Theorem (see Lemma~\ref{Gruenberg--Kegel theorem} in section~\ref{Prelim}), any graph with at least $7$ connected components is not isomorphic to the Gruenberg-Kegel graph of a finite group. To discuss the question of realizability of a graph as Gruenberg--Kegel graph of a group see, for example, papers~\cite{GKMK,Gors_Mas2,Gruber_etal,Maslova_Pagon}. However, the results which we obtain in section~\ref{Rec=>AS} allow to estimate an upper bound for a number of groups with the same Gruenberg--Kegel graph. In section~\ref{NumbRecBound}, we prove the following theorem.

\begin{theorem}\label{NumberPD} There exists a function $F(x)=O(x^7)$ such that for each labeled graph $\Gamma$ the following conditions are equivalent{\rm:}
\begin{itemize}
\item[$(1)$] there exist infinitely many groups $H$ such that $\Gamma(H)=\Gamma${\rm;}
\item[$(2)$] there exist more then $F(|V(\Gamma)|)$ groups $H$ such that $\Gamma(H)=\Gamma$, where $V(\Gamma)$ is the set of the vertices of $\Gamma$.
\end{itemize}
\end{theorem}

The estimate which we obtain in Theorem~\ref{NumberPD} for the function $F$  can certainly be improved. The following problem is of interest.

\begin{problem}\label{prob2}
Find the exact value for the function $F$, or at least a better upper bound.
\end{problem}

Recently M.~A.~Grechkoseeva and A~V.~Vasil'ev\footnote{M.~A.~Grechkoseeva, A.~V.~Vasil'ev, On the prime graph of a finite group with unique nonabelian composition factor, arXiv:2109.05860v1 [math.GR].} generalizing our ideas and using their new results on Gruenberg-Kegel graphs of finite groups with unique nonabelian composition factor have improved the upper bound in Theorem~\ref{NumberPD} to $O(x^5)$.

\smallskip

As a part of the solution of Problem~\ref{prob2}, the following problem arises.

\begin{problem} Find an improved upper bound for the number of almost simple groups with the same Gruenberg--Kegel graph.
\end{problem}

Note that by \cite{Zavarnitsine_2003,Zavarnitsine_2006_1}, there is no a constant $k$ such that for any almost simple group $G$, the number of pairwise non-isomorphic almost simple groups $H$ such that $\Gamma(G)=\Gamma(H)$ is at most $k$. However, if $G$ is simple, then A.~V.~Vasil'ev has conjectured that there are at most $4$ simple groups $H$ with $\Gamma(G)=\Gamma(H)$; see Problem~16.26 in~\cite{Kourovka}.

\medskip

We say that the group $G$ is {\it recognizable by isomorphism type of its Gruenberg--Kegel graph} if for each group~$H$, graphs $\Gamma(G)$ and $\Gamma(H)$ are isomorphic as abstract graphs if and only if $G \cong H$.

Since by A.~V.~Zavarnitsine~\cite{Zavarnitsine_2006}, the sporadic group $J_4$ is the unique group whose Gruenberg--Kegel graph has exactly $6$ connected components, we have that $J_4$ is recognizable by isomorphism type of its Gruenberg--Kegel graph. We construct some more examples of simple groups which are recognizable by isomorphism type of Gruenberg--Kegel graph: in section~\ref{UnlabGr}, we prove the following theorem.

\begin{theorem}\label{RecIsomType} Simple groups ${^2}G_2(27)$ and $E_8(2)$ are recognizable by isomorphism type of Gruenberg--Kegel graph. In particular, the group $E_8(2)$ is recognizable by its Gruenberg--Kegel graph.
\end{theorem}

It is easy to see that if a group $G$ is recognizable by isomorphism type of its Gruenberg--Kegel graph, then $G$ is recognizable by its Gruenberg--Kegel graph, therefore by Theorem~\ref{CoGrAlmSimple}, $G$ is almost simple. Thus, the following problem naturally arises.

\begin{problem} Let $G$ be an almost simple group. Decide whether $G$ is recognizable by isomorphism type of its Gruenberg--Kegel graph.
\end{problem}

Finally in this paper, we show how the Gruenberg--Kegel graph of a group gives information about various other graphs whose vertex sets are the elements of the group (so typically very much larger). The graphs we consider are the following (we give the adjacency rule for distinct elements $g,h\in G$ in each case):
\begin{itemize}
\item the \emph{commuting graph}~\cite{BF}: $gh=hg${\rm;}
\item the \emph{power graph}~\cite{KQ}: one of $g$ and $h$ is a power of the other{\rm;}
\item the \emph{enhanced power graph}~\cite{AACNS}: $\langle g,h\rangle$ is cyclic{\rm;}
\item the \emph{deep commuting graph}~\cite{CK}: the inverse images of $g$ and $h$ commute in every central extension of $G$.
\end{itemize}

In section~\ref{s:graphs}, we prove the following results.

\begin{theorem}\label{t:graphiso}
For a finite group $G$, let $\mathrm{T}(G)$ denote one of the above four types of graph on $G$. If $G$ and $H$ are groups with $\mathrm{T}(G)=\mathrm{T}(H)$, then the Gruenberg--Kegel graphs of $G$ and $H$ are equal.
\end{theorem}

\begin{theorem}
Let $G$ be a finite group. Then the following are equivalent{\rm:}
\begin{itemize}
\item[$(a)$] the enhanced power graph of $G$ is equal to the power graph{\rm;}
\item[$(b)$] the Gruenberg--Kegel graph of $G$ has no edges{\rm;}
\item[$(c)$] one of the following statements holds{\rm:}

$(1)$ $|\pi(G)|=1$ and $G$ is a $p$-group{\rm;}

$(2)$ $|\pi(G)|=2$ and $G$ is a {\rm(}solvable{\rm)} Frobenius group or $2$-Frobenius group{\rm;}

$(3)$ $|\pi(G)|=3$ and $G \in $ $\{A_6,$ $\PSL_2(7),$ $\PSL_2(17)$, $M_{10}\}${\rm;}

$(4)$ $|\pi(G)|=3$, $G/O_2(G)$ is $\PSL_2(2^n)$ for $n \in \{2,3\}$, and if $O_2(G)\not=\{1\}$, then $O_2(G)$ is the direct product of minimal normal subgroups of $G$, each of which is of order $2^{2n}$ and as a $G/O_2(G)$-module is isomorphic to the natural $\GF(2^n)SL_2(2^n)$-module.

$(5)$ $|\pi(G)|=4$ and $G \cong\PSL_3(4)$.

$(6)$ $|\pi(G)|=4$, $G/O_2(G)$ is $\Sz(2^n)$ for $n \in \{3,5\}$, and if $O_2(G)\not=\{1\}$, then $O_2(G)$ is the direct product of minimal normal subgroups of $G$, each of which is of order $2^{4n}$ and as a $G/O_2(G)$-module is isomorphic to the natural $\GF(2^n)\Sz(2^n)$-module of dimension $4$.
\end{itemize}
\label{t:nullgk}
\end{theorem}

\section{Preliminaries}\label{Prelim}

Let $q>1$ be a natural number and $r$ be an odd prime such that $(q, r) = 1$.
Denote by $e(r,q)$ the multiplicative order of $q$ modulo $r$, that is,
the minimal natural number $m$ such that $q^m \equiv 1 \pmod{r}$. For odd $q$ define $e(2, q) = 1$ if $q \equiv 1 \pmod{4}$ and $e(2, q) = 2$ if $q \equiv 3 \pmod{4}$.

\begin{lemma}[{\rm Zsigmondy's Theorem, see \cite{Zsigmondy}}]\label{Zhigmondy} Let $q>1$ be a natural number. For each $m$ there exists a prime $r$ such that $e(r, q) = m$, except the following cases{\rm:} $q = 2$  and $m = 1${\rm;} $q = 3$ and $m = 1${\rm;}  $q = 2$ and $m = 6$. In particular, $r$ divides $q^n - 1$ and doesn't divide $q^i - 1$ for $1 \le i \le n-1$, except for the following three cases{\rm:} $q = 2$ and $n = 6${\rm;} $q = 2^k - 1$ for some prime $k$ and $n = 2${\rm;} $q=2$ and $n=1$. \end{lemma}

In the notation of Lemma~\ref{Zhigmondy}, any prime $r$ which divides $q^n - 1$ and doesn't divide $q^i - 1$ for $1 \le i \le n-1$ is called a primitive prime divisor of the number $q^n-1$.  Note that a primitive prime divisor of a number $q^n-1$ can be defined non-uniquely.  For example, $11^3-1=2\times 5 \times 7\times 19$, $11^2-1=2^3\times 3 \times 5$, and $11-1=2\times 5$. Thus, primitive prime divisors of the number $11^3-1$ are the primes $7$ and $19$.

\begin{lemma}{\rm (see \cite{Gerono})}\label{Gerono} Let $p$ and $q$ be primes such that $p^a - q^b = 1$ for some integer numbers $a \ge 0$ and $b \ge 0$. Then $(p^a, q^b) \in \{(3^2, 2^3),(2^a, q),(p, 2^b)\}$, where $a$ is a prime and $b$ is a power of $2$.

\end{lemma}

\medskip

Our graph-theoretic and group-theoretic terminology is mostly standard; but
we list a few points here.

Let $\pi$ be a set of primes.
Given a natural number $n$, denote by $\pi(n)$ the set of its prime divisors. Then $\pi (|G|)$ is exactly $\pi(G)$ for any group $G$.
A natural number $n$ with $\pi(n) \subseteq \pi$ is called a $\pi$-number, and a group $G$ with $\pi(G) \subseteq \pi$ is called a $\pi$-group.

A {\it $n$-clique} (resp.\ a {\it $n$-coclique}) is a graph with $n$ vertices
in which all the vertices are pairwise adjacent (resp. non-adjacent).

If $G$ and $H$ are groups and $p$ is a prime,
then we will denote by $S(G)$ the {\it solvable radical} of $G$ (the largest solvable normal subgroup of $G$), by $F(G)$ the {\it Fitting subgroup} of $G$ (the largest nilpotent normal subgroup of $G$), by $\Phi(G)$ the {\it Frattini subgroup} of $G$ (the intersection of all maximal subgroups of $G$), and by $\Soc(G)$ the {\it socle} of $G$ (the subgroup of $G$ generated by the set of all non-trivial minimal normal subgroups of $G$). By $G.H$ we denote any extension of $G$ by $H$, by $G:H$ (or $G \rtimes H$) we denote a split extension (or semidirect product) of $G$ by $H$, by $O_p(G)$ the largest normal $p$-subgroup of $G$, by $O_{p'}$ the largest normal subgroup of $G$ whose order is not divisible by $p$, and by $O(G)$ the largest normal subgroup of odd order of $G$.
Denote the number of connected components of $\Gamma(G)$
by $s(G)$, and the set of connected
components of $\Gamma(G)$ by $\{\pi_i(G) \mid 1 \leq i \leq s(G) \}$; for a group $G$
of even order, we assume that $2 \in \pi_1(G)$. Denote by $t(G)$ the \emph{independence number} of $\Gamma(G)$ (the greatest cardinality of a coclique in $\Gamma(G)$), and by $t(r,G)$ the greatest cardinality of a coclique in $\Gamma(G)$ containing a prime $r$.

\medskip

\begin{lemma}[{\rm Gruenberg--Kegel Theorem, \cite[Theorem~A]{Williams}}]\label{Gruenberg--Kegel theorem} If~$G$ is a group with disconnected Gruenberg--Kegel graph, then one of the following statements holds{\rm:}
\begin{itemize}
\item[$(1)$] $G$ is a Frobenius group{\rm;}
\item[$(2)$] $G$ is a $2$-Frobenius group{\rm;}
\item[$(3)$] $G$ is an extension of a nilpotent $\pi_1(G)$-group by a group~$A$, where $S \unlhd A\le\Aut(S)$,~$S$ is a simple non-abelian group with $s(G)\le s(S)$, and $A/S$ is a $\pi_1(G)$-group.
\end{itemize}
\end{lemma}

\vspace{10mm}

\begin{table}[htbp]
{\footnotesize
\centerline{\begin{tabular}{|p{5mm}|p{12mm}|p{15mm}|p{23mm}|p{16mm}|p{16mm}|p{16mm}|p{11mm}|p{6mm}|}
\hline

& & & & & & & &\\ $s(S)$ &~~~$S$ &Restrictions&~~~~~~ $\pi_1(S)$ &~~~~~~$\pi_2(S)$
&~~~~ $\pi_3(S)$ &~~~ $\pi_4(S)$ & $\pi_5(S)$ & $\pi_6(S)$\\ & &~~~
& & & & & &
\\ \hline 4 & $A_2(4)$ & & $\{2\}$ & \{3\} & \{5\} & \{7\} & &\\[3ex]

& $^2B_2(q)$ & $q{=}2^{2m{+}1}{>}2 $ & $\{2\}$ & $\pi(q{-}1)$ & $\pi(q{-}\sqrt
{2q}{+}1)~$&$\pi(q{+}\sqrt {2q}{+}1)$ & &\\[3ex]

& $^2E_6(2)$ & & $\{2, 3, 5, 7, 11\}$ & \{13\} & \{17\} & \{19\} & &\\[3ex] & $E_8(q)$
&$q{\equiv} 2,3(5)$ & $\pi \big(q(q^8{-}1)(q^{12}{-}1)$ & $\pi(\frac{
q^{10}{+}q^5{+}1}{ q^2{+}q{+}1})$ & $\pi(q^8{-}q^4{+}1)$ & $\pi(\frac{
q^{10}{-}q^5{+}1}{ q^2{-}q{+}1})$ & &\\[0.3ex] & & & $(q^{14}{-}1)(q^{18}{-}1)$
& & & & &\\[0.5ex] & & & $(q^{20}{-}1)\big)$ & & & & &\\[3ex] & $M_{22}$ & &
$\{2, 3\}$ & \{5\} & \{7\} & \{11\} & &\\[3ex] &~$J_1$ & & $\{2, 3, 5\}$ & \{7\} & \{11\} & \{19\} &
&\\[3ex] & $O'N$ & & $\{2, 3, 5, 7\}$ & \{11\} & \{19\} & \{31\} & &\\[3ex] & $LyS$ & &
$\{2, 3, 5, 7, 11\}$ & \{31\} &\{37\} & \{67\} & &\\[3ex] & $Fi_{24}'$ & & $\{2, 3, 5, 7,
11, 13\}$ & \{17\} & \{23\} & \{29\} & &\\[3ex] &~$F_1=M$ & & $\{2, 3, 5, 7, 11, 13, $ & \{41\} &
\{59\} &\{71\}& &\\[0.5ex] & & & $17, 19, 23, 29$, & & & & &\\[0.5ex] & & & $31, 47\}$
& & & & &\\[3ex]

5 & $E_8(q)$ &$q{\equiv} 0,1,4(5)$ & $\pi (q(q^8{-}1)(q^{10}{-}1)$ &
$\pi(\frac{q^{10}{+}q^5{+}1}{q^2{+}q{+}1}$) &
$\pi(\frac{q^{10}{-}q^5{+}1}{q^2{-}q{+}1})$ & $\pi(q^8{-}q^4{+}1)$&
$\pi(\frac{q^{10}{+}1}{q^2{+}1})$ &\\[0.5ex]

& & & $(q^{12}{-}1)(q^{14}{-}1)$ & & & & &\\[0.5ex]

& & & $(q^{18}{-}1))$ & & & & &\\[3ex]

6 &~$J_4$ & & $\{2, 3, 5, 7, 11\}$ & \{23\} & \{29\} & \{31\} & \{37\} & \{43\}\\ \hline
\end{tabular}}
}
\caption{\label{table}Finite simple groups~$S$ with $s(S)>3$}
\end{table}

Note that in the Gruenberg--Kegel graph of a nonabelian simple group, the number $2$ is usually non-adjacent to at least one odd number~\cite{VasilVdov_2005}. The following generalization of the Gruenberg--Kegel Theorem was an important tool in investigations of recognizability of a group by spectrum.

\begin{lemma}[{\rm \cite[Propositions~2 and~3]{Vasil_2005}}]\label{VasAnd} Let $G$ be a non-solvable group such that $t(2,G) \ge 2$.
Then the following conditions hold{\rm:}
\begin{itemize}
\item[$(1)$] there exists a simple non-abelian group $S$ such that $S \unlhd G/S(G) \le\Aut(S)${\rm;}
\item[$(2)$] if $\rho \subseteq \pi(G)$ is a coclique in $\Gamma(G)$ with $|\rho|\ge 3$, then at most one of the primes from $\rho$ divides
the product $|S(G)|\cdot|G/S(G):S|$. In particular, $t(S) \ge t(G)-1${\rm;}
\item[$(3)$] one of the following conditions holds{\rm:}
\begin{itemize}
\item[$(a)$]  every $p \in \pi(G)$ which is non-adjacent to $2$ in $\Gamma(G)$ doesn't divide the product $|S(G)|\cdot|G/S(G):S|$. In particular, $t(2, S) \ge t(2, G)${\rm;}
\item[$(b)$] there exists $r \in \pi(S(G))$ which is non-adjacent to $2$ in $\Gamma(G)${\rm;} in this case $t(G) = 3$, $t(2, G) = 2$ and $S \cong   A_7$ or $\PSL_2(q)$ for any odd $q$.
\end{itemize}
\end{itemize}
\end{lemma}

To prove the main results of this paper, we need the following easily-proved assertions.

\begin{lemma}\label{GensNumber}
Let $G$ be a group with a normal series of length $r$ with cyclic factors.
Then any subgroup of $G$ can be generated by at most $r$ elements.
\end{lemma}

\begin{proof}
Let
\[G=G_0>G_1>\cdots>G_r=1\]
be the normal series with cyclic factors. Let $H$ be an arbitrary subgroup
of $G$. For each $i$, $G_{i+1}\cap H$ is a normal subgroup of $G_i\cap H$,
with factor group
\[(G_{i-1}\cap H)/(G_i\cap H)\cong (G_{i-1}\cap H)G_i/G_i\le G_{i-1}/G_i\]
which is cyclic; let $h_i$ be an element of $H$ such that
$(G_i\cap H)h_i$ generates this quotient.

We claim that $h_1,\ldots,h_r$ generate $H$. This is proved by backward
induction. For $i=r$, we have that $h_r$ generates $H\cap G_{r-1}$. Suppose
that $h_{i+1},\ldots,h_r$ generate $H\cap G_i$. Then it is clear from the
above that $h_i,\ldots,h_r$ generate $H\cap G_{i-1}$. So the induction goes
through, and the claim is proved on putting $i=1$. \hfill$\Box$

\end{proof}

\begin{lemma}\label{SupplNumber}
Let $G$ be a group with a cyclic normal subgroup $C$ such that $G/C$ is $k$-generated.
Then the number of pairwise distinct supplements of $C$ in $G$ {\rm(}subgroups $H$ such that $G=HC${\rm)} is at most $|C|^{k+1}$.
\end{lemma}

\begin{proof}

Let $\overline{g}_1, \ldots, \overline{g}_k$ be generators of $G/C$, and $g_1, \ldots, g_k$ be their preimages in $G$.
Define $C_i=Cg_i$ to be the corresponding right cosets. It is clear that $|C_i|=|C|$ for each $i$.

Now let $H$ be a supplement to $C$ in $G$. Consider $H_1=H\cap C$. Since $C$ is cyclic, we have $H_1=\langle h \rangle$ is cyclic.
Since $G=HC$, we have $H/(H\cap C) \cong G/C$. Let $h_1, \ldots, h_k$ be preimages in $H$ of the elements $\overline{g}_1, \ldots, \overline{g}_k$ from $H/(H\cap C) \cong G/C$.

It is clear that $h \in C$ and $h_i \in C_i$ for each $i$. We claim that $$H = \langle h, h_1, \ldots, h_k\rangle.$$
It is easy to see that $ h, h_1, \ldots, h_k \in H$, therefore, $ \langle h, h_1, \ldots, h_k\rangle \le H$.
Let $K = \langle h, h_1, \ldots, h_k\rangle$. Show that $|K| \ge |H|$ and, therefore, $K=H$. Indeed, $h \in K$, therefore, $K \cap C \ge H \cap C =\langle h \rangle$, and $K/(K \cap C) \ge \langle \overline{g}_1, \ldots, \overline{g}_k \rangle=G/C$. Thus, $|K|\ge |H\cap C|\cdot |G/C|=|H|$ and therefore, $|K|=|H|$.

So, we have proved that each supplement $H$ to $C$ in $G$ can be generated by an element $h \in C$ and elements $h_1, \ldots h_k$ such that $h_i \in C_i$ for each $i$. It is easy to see that there are at most $|C|$ possibilities to chose $h$ and at most $|C_i|=|C|$ possibilities to chose each $h_i$. Thus, there are at most $|C|^{k+1}$ pairwise distinct supplements to $C$ in $G$. \hfill$\Box$

\end{proof}

\begin{lemma}\label{LieType} Let $\pi$ be a finite set of primes, and $S=G_n(q)$, where $q=p^l$, be a simple group of Lie type of Lie rank $n$ with base field $GF(q)$ such that $\pi(S)\subseteq \pi$. Then the following statements hold{\rm:}
\begin{itemize}
\item[$(1)$] there are at most $|\pi|$ choices for $p${\rm;}
\item[$(2)$] there are at most $|\pi|+1$ choices for $l${\rm;}
\item[$(3)$] $d(l) \le |\pi|+1$, where $d(l)$ is the number of pairwise distinct divisors of $l${\rm;}
\item[$(4)$] If $S$ s a classical group, then $n \le 2|\pi|+3$.
\end{itemize}
\end{lemma}

\begin{proof} It is clear that $p$ divides $|S|$, therefore, $p \in \pi$ and $(1)$ holds.

Note that $q-1$ divides $|S|$, and by Lemma~\ref{Zhigmondy}, excluding the cases $q=2$, $q=2^6$, and $q$ is a Mersenne prime, the number $p^l-1$ has a primitive prime divisor $r \in \pi(S)\setminus \{p\}$. Note that there are at most $|\pi(S)|-1$ choices for $r$ and, therefore, at most $|\pi(S)|-1 < |\pi|$ choices for $l$ (or at most $|\pi(S)|+1 \le |\pi|+1$ or $|\pi(S)| \le |\pi|$ choices for $l$ if $p=2$ or $p$ is a Mersenne prime, respectively). Thus, $(2)$ holds.

Note that if $d(l)$ is the number of pairwise distinct divisors of $l$, then by Lemma~\ref{Zhigmondy}, there exist at least $d(l)-2$ pairwise distinct primes dividing $p^l-1$, and $p$ and all these primes are in $\pi$. Thus, $d(l)\le |\pi|+1$ and $(3)$ holds.

If $S=A_n(q)$, then  there is a divisor of $|S|$ of the form $q^m-1$ for each $1 \le m \le n+1$, therefore, with possible one or two exceptions due to Lemma~\ref{Zhigmondy}, the primitive prime divisor of the number $q^m-1$ for each $m \le n+1$ must lie in $\pi$. For each remainder family of classical groups, there is a divisor of $|S|$ of the form $q^{2m}-1$ for each $m \le m_0$, where $m_0$ grows linearly with the rank. In the five families $B_n(q)$, $C_n(q)$, $D_n(q)$, ${}^2A_n(q)$ and ${}^2D_n(q)$, we have $m_0=n$, $m_0=n$, $m_0=n-1$, $m_0=\lfloor n/2\rfloor$, and $m_0=n-1$, respectively. Now a primitive prime divisor of the number $q^{2m}-1$ for each $m \le m_0$ with possible one or two exceptions due to Lemma~\ref{Zhigmondy} must lie in $\pi$. Thus, at least $n \le 2|\pi|+3$ in each case  and $(4)$ holds. \hfill$\Box$

\end{proof}

\begin{lemma}{\rm (see \cite[Propositions~2.5,~3.2,~4.5]{VasilVdov_2005} and \cite[Propositions~2.7]{VasilVdov_2011})}\label{AdjCritE8}
Let $G\cong E_8(q)$, where $q$ is a power of $p=2$. Let $r, s \in \pi(G)$ and $r\not = s$. Then $r$ and $s$ are non-adjacent in $\Gamma(G)$ if and only if one of the following conditions holds{\rm:}

$(1)$ $2=p \not \in \{r,s\}$, $1\le e(r,q) <e(s,q)$, and one of the following conditions holds{\rm:}

$\mbox{ }\mbox{ }\mbox{ }\mbox{ }\mbox{ }$ $(1a)$ $e(s,q) = 6$ and $e(r,q) = 5${\rm;}

$\mbox{ }\mbox{ }\mbox{ }\mbox{ }\mbox{ }$ $(1b)$ $e(s,q) \in \{7, 14\}$ and $e(r,q) \ge  3${\rm;}

$\mbox{ }\mbox{ }\mbox{ }\mbox{ }\mbox{ }$ $(1c)$ $e(s,q) = 9$ and $k \ge 4${\rm;}

$\mbox{ }\mbox{ }\mbox{ }\mbox{ }\mbox{ }$ $(1d)$ $e(s,q) \in \{8, 12\}$, $e(r,q) \ge 5$, and $ e(r,q) \not = 6${\rm;}

$\mbox{ }\mbox{ }\mbox{ }\mbox{ }\mbox{ }$ $(1e)$ $e(s,q) = 10$, $e(r,q) \ge 3$, and $e(r,q) \not = 4,6${\rm;}

$\mbox{ }\mbox{ }\mbox{ }\mbox{ }\mbox{ }$ $(1f)$ $e(s,q) = 18$ and $e(r,q)\not \in \{ 1, 2, 6\}${\rm;}

$\mbox{ }\mbox{ }\mbox{ }\mbox{ }\mbox{ }$ $(1g)$ $e(s,q) = 20$ and $r \cdot e(r,q) \not = 20${\rm;}

$\mbox{ }\mbox{ }\mbox{ }\mbox{ }\mbox{ }$ $(1h)$ $e(s,q) \in \{15, 24, 30\}${\rm;}

$(2)$ $r=p=2$ and $ e(s,q) \in  \{15, 20, 24, 30\}$.

\end{lemma}

\begin{lemma}\label{E8(2)} Let $G=E_8(2)$. Then the following statements hold{\rm:}

$(1)$ $|\pi(G)|=16$ and $\pi(G)=\{2$, $3$, $5$, $7$, $11$, $13$, $17$, $19$, $31$, $41$, $43$, $73$, $127$, $151$, $241$, $331\}$.

$(2)$ If $S=E_8(q)$ for $q>2$, then $|\pi(S)|>|\pi(G)|$.

$(3)$ Adjacency in $\Gamma(G)$ is presented in Table~{\rm\ref{table2}}. In particular, $s(G)=4$, $|\pi_1(G)|=13$, and $|\pi_i(G)|=1$ for $i \in \{2,3,4\}$.

\end{lemma}

\begin{table}[htbp]
{\footnotesize
\centerline{\begin{tabular}{|p{20mm}|p{10mm}|p{50mm}|}
\hline
$\mbox{Vertex }x$ & $\mbox{Degree}$ & $\mbox{Neighbors of }x \mbox{ in }\Gamma(E_8(2))$\\ \hline

 $2$& $11$ & $3$, $5$, $7$, $11$, $13$, $17$, $19$, $31$, $43$, $73$, $127$ \\ \hline

 $3$ & $11$ & $2$, $5$, $7$, $11$, $13$, $17$, $19$, $31$, $43$, $73$, $127$\\ \hline

 $5$ & $8$ & $2$, $3$, $7$, $11$, $13$, $17$, $31$, $41$ \\   \hline

 $7$ & $7$ & $2$, $3$, $5$, $13$, $17$, $31$, $73$\\ \hline

 $11$ & $3$ & $2$, $3$, $5$ \\ \hline

 $13$ & $4$ & $2$, $3$, $5$, $7$ \\ \hline

 $17$ & $4$ & $2$, $3$, $5$, $7$ \\ \hline

 $19$ & $2$ & $2$, $3$ \\ \hline

 $31$ & $4$ & $2$, $3$, $5$, $7$ \\ \hline

 $41$ & $1$ & $5$ \\ \hline

 $43$ & $2$ & $2$, $3$ \\ \hline

 $73$ & $3$ & $2$, $3$, $7$ \\ \hline

 $127$ & $2$ & $2$, $3$ \\ \hline

 $151$ & $0$ & \\ \hline

 $241$& $0$ & \\ \hline

 $331$&  $0$ & \\ \hline

\end{tabular}

}}
\caption{\label{table2}Adjacency of vertices in $\Gamma(E_8(2))$}
\end{table}

\begin{proof} $(1)$ See, for example, \cite{Atlas}.

$(2)$ We have $|S|=q\cdot \prod_{i \in \{2,8,12, 14, 18, 20, 24, 30\}} (q^i-1)$, therefore, by Lemma~\ref{Zhigmondy}, $|\pi(S)|\ge 17 >|\pi(G)|$.

$(3)$ Note that $e(3,2)=2$, $e(5,2)=4$, $e(7,2)=3$, $e(11,2)=10$,
$e(13,2)=12$,
$e(17,2)=8$,
$e(19,2)=18$,
$e(31,2)=5$,
$e(41,2)=20$,
$e(43,2)=14$,
$e(73,2)=9$,
$e(127,2)=7$,
$e(151,2)=15$,
$e(241,2)=24$,
$e(331,2)=30$. Now Lemma~\ref{AdjCritE8} gives adjacency in $\Gamma(G)$. \hfill$\Box$

\end{proof}

\begin{lemma}{\rm (see \cite{EAtlas})}\label{F1=M} Let $G$ be the sporadic group $M$. Then the following statements hold{\rm:}

$(1)$ $|\pi(G)|=15$ and $\pi(G)=\{2$, $3$, $5$, $7$, $11$, $13$, $17$, $19$, $23$, $29$, $31$, $41$, $47$, $59$, $71\}$.

$(2)$ Adjacency in $\Gamma(G)$ is presented in Table~{\rm\ref{table3}}.

\end{lemma}

\begin{table}[htbp]
\centerline{\begin{tabular}{|p{20mm}|p{20mm}|p{45mm}|}
\hline
$\mbox{Vertex }x$ & $\mbox{Degree}$ & $\mbox{Neighbors of }x \mbox{ in }\Gamma(M)$\\ \hline

 $2$&  $10$ & $3$, $5$, $7$, $11$, $13$, $17$, $19$, $23$, $31$, $47$\\ \hline

 $3$ &  $10$ & $2$, $5$, $7$, $11$, $13$, $17$, $19$, $23$, $29$, $31$\\ \hline

 $5$ &  $5$ & $2$, $3$, $7$, $11$, $19$\\ \hline

 $7$ & $4$ & $2$, $3$, $5$, $17$ \\ \hline

 $11$ & $3$ & $2$, $3$, $5$\\ \hline

 $13$ & $2$ & $2$, $3$ \\ \hline

 $17$ & $3$ & $2$, $3$, $7$ \\ \hline

 $19$ & $3$ & $2$, $3$, $5$\\ \hline

 $23$ & $2$ & $2$, $3$ \\ \hline

 $29$ & $1$ & $3$\\ \hline

 $31$ & $2$ & $2$, $3$ \\ \hline

 $41$ & $0$ & \\ \hline

 $47$ & $1$ & $2$\\ \hline

 $59$ & $0$ & \\ \hline

 $71$& $0$ & \\ \hline

\end{tabular}

}
\caption{\label{table3}Adjacency of vertices in $\Gamma(M)$}
\end{table}

\begin{lemma}{\rm (see \cite[Lemma~14]{Maslova_Pagon})}\label{K(1,5)}
If $G$ is a group such that $\Gamma(G)$ is a bipartite graph with parts of sizes $1$ and $5$, then $\pi(G)=\{2,3,7,13,19,37\}$ and $G/O_2(G) \cong {^2}G_2(27)$.
\end{lemma}

\begin{lemma}{\rm (see \cite[Theorem~A]{Zavarnitsine_2006})}\label{Rec2G2}
Groups ${^2}G_2(3^{2m+1})$  for $m \ge 1$ are recognizable by Gruenberg--Kegel graph.
\end{lemma}

\begin{lemma}{\rm (see \cite[Lemma 1]{Mazurov_1997_1})}\label{OrdersInExt} Let $G$ be a group, $N$ its normal subgroup such that $G/N$
a Frobenius group with kernel $F$ and cyclic complement $C$. If $(|F|,|N|) = 1$
and $F \not \le NC_G(N)/N$, then $s\cdot|C| \in  \omega(G)$ for any $s \in \pi(N)$.

\end{lemma}

\begin{lemma}{\rm (see \cite[Lemma~10]{Zavarn_Mazurov})}\label{SpecSemDirProd} Let $V$ be a normal elementary abelian subgroup of a
group $G$. Put $H = G/V$ and denote by $G_1 = V \rtimes H$ the natural semidirect product. Then $\omega(G_1) \subseteq \omega(G)$.
\end{lemma}

Let $S$ be a simple group of Lie type in characteristic $p$. Let $A$ be any abelian $p$-group with an $S$-action.
Any element $s \in S$ is said to be {\it unisingular} on $A$ if $s$ has a (non-zero) fixed point on $A$.
$S$ is said to be {\it unisingular} if every element $s \in S$ acts unisingularly on every finite
abelian $p$-group $A$ with an $S$-action.

\begin{lemma}{\rm (see \cite[Theorem~1.3]{Guralnick_Tiep})}\label{Unisingular} If $G =E_8(q)$ with q arbitrary, then $G$ is unisingular.
\end{lemma}

\begin{lemma}{\rm (see \cite[Proposition~2]{Zavarnitsine_2013})}\label{3D4} Let $G = {^3}D_4(q)$, where $q$ is a power of a prime $p$.
If $G$ acts on a non-trivial vector space $V$ over a field of characteristic distinct from $p$, then
each element $x \in G$ of order $q^4 - q^2 + 1$ fixes a non-zero vector $v(x) \in V$.
\end{lemma}

\section{Proofs of Theorems~\ref{UnrecGK} and~\ref{CoGrAlmSimple}}\label{Rec=>AS}

To prove the main results of this paper, for a given finite set $\pi$ of prime numbers, we need to estimate the number of simple groups $S$ such that $\pi(S)\subseteq\pi$. The following proposition is straightforward, and the main idea of its proof can be found, for example, in~\cite{Mazurov_1994_1} (see the remark after Lemma~2 in~\cite{Mazurov_1994_1}).

\begin{prop}\label{NumSimple} Let $\pi$ be a finite set of primes. The number of pairwise non-isomorphic non-abelian simple groups $S$ with $\pi(S)\subseteq\pi$ is finite, and is at most $O(|\pi|^3)$.
\end{prop}

\begin{proof}  Following the classification of finite simple groups, we divide into the following cases.

\smallskip
{\it Case $S$ sporadic{\rm:}} There are clearly at most $26$ such groups.

\smallskip
{\it Case $S$ alternating{\rm:}} The alternating group $A_m$ has order
divisible by all primes less than $m$. So, if $\pi(A_m)\subseteq\pi$,
then $m$ does not exceed the $(|\pi|+1)$st prime number $p_{|\pi|+1}$. By the Prime Number Theorem, $p_{|\pi|+1}$ is roughly $|\pi|\log |\pi|$. So the number of alternating groups does not exceed this number.

\smallskip

{\it Case $S$ of Lie type{\rm:}} These groups fall into six families
$A_n(p^l)$, $B_n(p^l)$, $C_n(p^l)$, $D_n(p^l)$, ${}^2A_n(p^l)$ and ${}^2D_n(p^l)$
parametrised by rank (one parameter $n$) and field order (two parameters $p$ and $l$), and ten families $E_6(p^l)$, $E_7(p^l)$,
$E_8(p^l)$, $F_4(p^l)$, $G_2(p^l)$, ${}^2E_6(p^l)$, ${}^3D_4(p^l)$, ${}^2F_4(2^l)$,
${}^2B_2(2^l)$ and ${}^2G_2(3^l)$ parametrised by field order (two parameters $p$ and $l$ and for the
last three, only one parameter $l$).

By Lemma~\ref{LieType}, there are at most $|\pi|$ choices for characteristic $p$ (except for one-parameter families), at most $|\pi|+1$ choices for $l$, and for classical groups at most $2|\pi|+3$ choices for their ranks. Thus, there are at most $O(|\pi|)$ groups in each of the one-parameter families, $O(|\pi|^2)$ groups in each of the two-parameter families, and at most $O(|\pi|^3)$ groups in each of the three-parameter families.

Thus, we conclude that there are at most $O(|\pi|^3)$ simple groups $S$ of Lie type such that $\pi(S)\subseteq \pi$. \quad$\Box$

\end{proof}

\medskip

\noindent {\bf Proof of Theorem~\ref{UnrecGK}} We show first that $(2) \Rightarrow (1)$. Assume that a group $G$ contains a non-trivial solvable normal subgroup. Then by Theorem~\ref{UnrecSpec}, there exist infinitely many groups $H$ such that $\omega(G)=\omega(H)$, and Gruenberg--Kegel graphs of all these groups coincide with $\Gamma(G)$.

Now we show that $(1) \Rightarrow (2)$. Assume that $G$ is a group such that there exist infinitely many groups $H$ such that $\Gamma(G)=\Gamma(H)$. Assume that for each group $H$ with $\Gamma(G)=\Gamma(H)$, the solvable radical of $H$ is trivial (otherwise the statement of the theorem holds trivially). If $2$ is adjacent to each odd vertex of $\Gamma(G)$, then $\Gamma(G)=\Gamma(C_2\times G)$, where $C_2$ is the cyclic group of order $2$, and $S(C_2\times G)\not =1$, a contradiction. Thus, for each $H$ with $\Gamma(G)=\Gamma(H)$, $t(2,H)=t(2,G)\ge 2$ and, therefore, by Lemma~\ref{VasAnd}, $\Soc(H)$ is a non-abelian simple group such that $\pi(\Soc(H)) \subseteq \pi(G)$.  By Proposition~\ref{NumSimple}, for each finite set $\pi$ of primes, the number of simple groups $T$ such that $\pi(T)\subseteq \pi$ is finite, therefore the number of almost simple groups $H$ with $\pi(\Soc(H)) \subseteq \pi(G)$ is finite, a contradiction.\hfill$\Box$

\medskip

\noindent {\bf Proof of Theorem~\ref{CoGrAlmSimple}}  Note that $G$ is non-solvable, since otherwise, by Theorem~\ref{UnrecGK}, there are infinitely many  groups $H$ with $\Gamma(G)=\Gamma(H)$. Moreover, in $\Gamma(G)$, each vertex $r$ is non-adjacent to at least one other vertex; otherwise $\Gamma(G)=\Gamma(G\times C_r)$ and again by Theorem~\ref{UnrecGK}, there are infinitely many groups $H$ with $\Gamma(G)=\Gamma(H)$. In particular, $t(2,G)\ge 2$.

By Lemma~\ref{VasAnd}, each group $H$ with $\Gamma(G)=\Gamma(H)$ is such that $H/S(H)$ is almost simple. Now again if $S(H)\not =1$ for some $H$ with $\Gamma(G)=\Gamma(H)$, then by Theorem~\ref{UnrecGK}, there are infinitely many groups $H$ with $\Gamma(G)=\Gamma(H)$. Therefore, each group $H$ with $\Gamma(G)=\Gamma(H)$ is almost simple, in particular, $G$ is almost simple.

By Lemmas~7--14 from~\cite{Gors_Mas2}, if $G$ is almost simple and $t(G)<3$, then there exists a solvable group $T$ such that $\Gamma(G)=\Gamma(T)$. Therefore in this case $G$ is unrecognizable by Gruenberg--Kegel graph. \hfill$\Box$

\section{Proof of Theorem~\ref{NumberPD}}\label{NumbRecBound}

Let $\pi$ be a finite set of primes. We must show that there is a polynomial function $F$ such that the number of almost simple groups $G$ such that $\pi(G)\subseteq \pi$ is at most $F(|\pi|)$.

By Proposition~\ref{NumSimple}, the number of simple groups $S$ such that $\pi(S)\subseteq \pi$ is bounded by a polynomial function of $|\pi|$, moreover the function is at most $O(|\pi|^3)$. Thus, it is sufficient to show that for each simple group $S$ such that $\pi(S)\subseteq \pi$, the number of almost simple groups with socle $S$ is bounded by a polynomial function of $|\pi|$. This is clear if $S$ is sporadic or alternating. Thus, it is sufficient to consider the case when $S$ is a group of Lie type. Since there is a one-to-one correspondence between almost simple groups with socle $S$ and subgroups of $\Out(S)$, we need to prove the following result:

\begin{prop}\label{OutSimple}
Let $\pi$ be a finite set of primes and $S$ be a simple group of Lie type such that $\pi(S)\subseteq \pi$. Then there is a polynomial function $f$ such that the number of subgroups of $\Out(S)$ is bounded by $f(|\pi(S)|)$, and it is at most $O(|\pi|^4)$.
\end{prop}

\begin{proof} Let $S=G_n(q)$, where $q=p^l$, be a simple group of Lie type of Lie rank $n$ with base field $GF(q)$ such that $\pi(S)\subseteq \pi$.

Then it is well-known (see, for example \cite[Theorem~2.5.12]{Gorenstein_LS}) that
\[\Out(S)=C\rtimes (G_1 \times G_2),\]
where $C=\mathrm{Outdiag}(S)$ and either $|C| \le 4$ or $S$ is of type $A_n(q)$ or ${^2}A_n(q)$, $C$ is cyclic, and $|C|$ is bounded by $n+1$;
$G_1$ is cyclic and $|G_1| \in \{l, 2l, 3l\}$;
and $G_2$ has order at most $2$ except in the case when $S$ is of type $D_4$, when it is isomorphic to $S_3$\footnote{Here it is sufficient that $\Out(S)=C.G_2.G_1$, where groups $C$, $G_1$, and $G_2$ are as above.}.

Let $N(G)$ be the set of all the subgroups of a group $G$ with a normal subgroup $C$. We can define an equivalence relation $\varrho_C$ on $N(G)$ by setting $H_1 \varrho_C H_2$ if $H_1C/C=H_2C/C$. Then $N(G)$ is the union of equivalence classes of $\varrho_C$; that is, in our case,
\begin{equation}\label{lll}
N(\Out(S))=\bigcup\limits_{T \le G_1\times G_2}^{}\{H \mid HC/C=T\}.
\end{equation}

Since $G_1$ is cyclic, there is a one-to-one correspondence between the number of subgroups of $G_1$ and the number of pairwise distinct divisors of $|G_1|$. Since by Lemma~\ref{LieType}, $d(l)\le |\pi|+1$, it is easy to see that the number of subgroups of $G_1$ is bounded by a linear function of $|\pi|$. Now by Lemmas~\ref{GensNumber} and~\ref{SupplNumber}, the number of subgroups of $G_1\times G_2$ is bounded by a linear function of $|\pi|$. Thus, the number of classes in the union (\ref{lll}) is bounded by a linear function of $|\pi|$.

By Lemma~\ref{GensNumber}, each subgroup $T$ of the group $\Out(S)/C\cong G_1\times G_2$ is at most $2$-generated except in the case when $S$ is of type $D_4$, when $T$ is at most $3$-generated. Therefore by Lemma~\ref{SupplNumber}, if $S$ is not of type $D_4$, then the number of subgroups $H$ of $Out(S)$ such that $HC/C=T$ is at most $|C|^3 \le n^3 \le (2|\pi|+3)^3$; in the case of groups of type $D_4$ we have $|C| \le 4$ and possibly need to apply Lemma~\ref{SupplNumber} twice. Thus, the number of classes in the union (\ref{lll}) is bounded by a linear function of $|\pi|$ and the number of subgroups in each class is bounded by a polynomial function of $|\pi|$ which is at most $O(|\pi|^3)$, and therefore, $|N(\Out(S))|$ is bounded by a polynomial function of $|\pi|$, and it is at most $O(|\pi|^4)$. \hfill$\Box$

\end{proof}

\bigskip

Combining the results of Propositions~\ref{NumSimple} and~\ref{OutSimple}, we have proved the following assertion.

\begin{theorem} There exists a polynomial function $F$ such that if $\pi$ is a finite set of primes, then there are at most $F(|\pi|)$ pairwise non-isomorphic almost simple groups $G$ such that $\pi(G)\subseteq \pi$, and this number is at most $O(|\pi|^7)$.
\end{theorem}

So the following corollary completes the proof of Theorem~\ref{NumberPD}{\rm:}

\begin{cor} Let $G$ be a group such that there are more than $F(|\pi(G)|)$ pairwise non-isomorphic groups $H$ such that $\Gamma(H)=\Gamma(G)$. Then $G$ is unrecognizable by Gruenberg--Kegel graph.
\end{cor}

\begin{proof} Suppose that there is only a finite number of groups $H$ such that $\Gamma(H)=\Gamma(G)$. Then by Theorem~\ref{CoGrAlmSimple}, each group $H$ with $\Gamma(H)=\Gamma(G)$ is almost simple. Thus, there are more than $F(|\pi(G)|)$ pairwise non-isomorphic almost simple groups $H$ with $\pi(H)\subseteq \pi(G)$, a contradiction. \hfill$\Box$
\end{proof}

\section{Proof of Theorem~\ref{RecIsomType}}\label{UnlabGr}

If $G={^2}G_2(27)$, then $\Gamma(G)$ is as follows:
\medskip

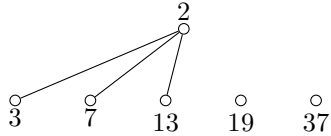
\centerline{
    \begin{tikzpicture}
    \tikzstyle{every node}=[draw,circle,fill=white,minimum size=4pt,
                            inner sep=0pt]
    \draw (0,0) node (3) [label=below:$3$]{}
         ++ (0:1.0cm) node (7) [label=below:$7$]{}
         ++ (0:1.0cm) node (13) [label=below:$13$]{}
     ++ (0:1.0cm) node (19) [label=below:$19$]{}
        ++ (0:1.0cm) node (37) [label=below:$37$]{}
        ++ (-1.0cm:-2.0cm) node (2) [label=above:$2$]{}
        (2)--(3)
        (2)--(7)
        (2)--(13);
    \end{tikzpicture}
}
 \medskip
\noindent If $H$ is a group such that $\Gamma(H)$ and $\Gamma(G)$ are isomorphic as abstract graphs, then $\Gamma(H)$ is a bipartite graph with parts of sizes $1$ and $5$. By Lemma~\ref{K(1,5)}, $\pi(H)=\{2,3,7,13,19,37\}$ and $H/O_2(H) \cong {^2}G_2(27)$. In particular,
$\Gamma(H)=\Gamma(G)$. But $G$ is recognizable by Gruenberg--Kegel graph by Lemma~\ref{Rec2G2}. Thus, $H\cong G$.

\medskip
Let $G=E_8(2)$ and $H$ be a group such that $\Gamma(H)$ and $\Gamma(G)$ are isomorphic as abstract graphs. By Lemma~\ref{E8(2)}$(3)$, $\Gamma(H)$ is disconnected. Now by Lemmas~\ref{Gruenberg--Kegel theorem} and~\ref{VasAnd}, $H/F(H)$ is almost simple, moreover, if $S=Soc(H/F(H))$, then $s(S) \ge s(G) \ge 4$. By Table~$1$ and Lemma~\ref{E8(2)}$(2)$, $S$ is one of the following groups{\rm:}  $A_2(4)\cong \PSL_3(4)$, ${^2}B_2(q)$, $^2E_6(2)$, $E_8(2)$, $M_{22}$, $J_1$, $O'N$, $LyS$, $Fi_{24}'$, $F_1=M$, $J_4$.

If $S$ is one of groups $A_2(4)$, $^2E_6(2)$, $M_{22}$, $J_1$, $O'N$, $LyS$, $Fi_{24}'$, $J_4$, then by \cite{Atlas}, $|\pi(\Aut(S)|\le 10$, therefore, $$|\pi(F(H)) \setminus \pi(H/F(H))|\ge 6.$$ Note that the primes from $\pi(F(H))$ form a clique in $\Gamma(H)$, therefore, $\Gamma(H)$ contains at least $6$ vertices of degree at least $5$, a contradiction to Lemma~\ref{E8(2)}$(3)$.

Let $S \cong {^2}B_2(q)$ for some $q$. Since $s(G)=s(S)$, we have $|\pi_i(S)|=1$ for $i \in \{2,3,4\}$. In particular, $q-1$ is a prime power. By Lemma~\ref{Gerono}, $q=2^a$, where $a$ is a prime. Therefore, $|\pi(H/F(H))| \le |\pi(\Aut(S)|=5$. Again we obtain that $$|\pi(F(H)) \setminus \pi(H/F(H))|\ge 11,$$ therefore, $\Gamma(H)$ contains at least $11$ vertices of degree at least $10$, a contradiction to Lemma~\ref{E8(2)}$(3)$.

Let $S\cong M$. Since $\Aut(M)\cong M$ and $|\pi(M)|=15$, there is $p \in \pi(F(H)) \setminus \pi(H/F(H))$. Note that by Lemma~\ref{Gruenberg--Kegel theorem}, $p \in \pi_1(H)$. Since $s(G)=s(S)$, we have $|\pi_i(S)|=1$ for $i \in \{2,3,4\}$. Thus, by Lemma~\ref{F1=M}, the primes $41$, $59$, and $71$ are non-adjacent to $p$ in $\Gamma(H)$. By \cite{Atlas,EAtlas}, $M$ contains subgroups $M_1 \cong 41:40$, $M_2 \cong 59:29$, $M_3\cong 71:35$, and $M_4 \cong 23:11$ which are Frobenius groups with cyclic complements. By Lemma~\ref{OrdersInExt}, the prime $p$ is adjacent in $\Gamma(H)$ to the primes $2$, $5$, $7$, $29$, and to at least one of the primes $11$ and $23$. Thus, degree of vertex $p$ in $\Gamma(H)$ is at least $5$ as well as by Lemma~\ref{F1=M}, degree of vertex $5$ in $\Gamma(H)$ is at least $6$, and degree of vertex $7$ in $\Gamma(H)$ is at least $5$. Moreover, by Lemma~\ref{F1=M}, degree of vertex $2$ in $\Gamma(H)$ is at least $11$ and degree of vertex $3$ in $\Gamma(H)$ is at least $10$. Thus, $\Gamma(H)$ contains at least $5$ vertices of degree at least $5$, a contradiction to Lemma~\ref{E8(2)}$(3)$.

Thus, $S \cong G=E_8(2)$, therefore $\Gamma(H)=\Gamma(G)$, $\pi(F(H))\subseteq \pi_1(G)$, and $H/F(H) \cong G$. Let $F(H)\not = 1$ and $r \in \pi(F(H)) \subseteq \pi_1(G)$. We show that $r$ and $241$ are adjacent in $\Gamma(H)$. Since $\Gamma(H)=\Gamma(S)=\Gamma(H/(O_{r'}(H)\times \Phi(O_r(H))))$, without loss of generality, we can assume that $F(H)=O_r(H)$ and $O_r(H)$ is elementary abelian. By Lemma~\ref{SpecSemDirProd}, we can assume that $H = O_r(H) \rtimes E_8(2)$.  If $r=2$, then by Lemma~\ref{Unisingular}, each element from $E_8(2)$ has a non-zero fixed point on $O_2(H)$. In particular, $2$ and $241$ are adjacent in $\Gamma(H)$. Thus, $r \not =2$. By \cite{Atlas}, the group $E_8(2)$ has a subgroup isomorphic to ${^3}D_4(4)$. Now by Lemma~\ref{3D4}, each element of order $241$ has a non-zero fixed point on $O_r(H)$. Thus, in any case $s(H)\le 3$, a contradiction. \hfill$\Box$

\section{Other graphs}
\label{s:graphs}

In this section we give the proofs of Theorems~\ref{t:graphiso} and \ref{t:nullgk}.

\paragraph{Proof of Theorem~\ref{t:graphiso}}
For each of the four possible types of graph, $G$ and $H$ have the same order, so their Gruenberg--Kegel graphs have the same set of vertices.

We show that in all cases except the power graph, primes $p$ and $q$ are adjacent in the Gruenberg--Kegel graph of $G$ if and only if there is a maximal clique in the graph $\mathrm{T}(G)$ with size divisible by $pq$. This is clear in the cases of the enhanced power graph and the commuting graph; for the maximal cliques in these are maximal cyclic subgroups or maximal abelian subgroups of $G$ respectively, and if their order is divisible by $pq$ (where $p$ and $q$ are distinct primes), then they contain elements of order $pq$. Conversely an element of order $pq$ is contained in a maximal cyclic (or abelian) subgroup.

Consider the deep commuting graph of a group $G$. It is shown in~\cite{CK} that this graph is the projection onto $G$ of the commuting graph of a \emph{Schur cover} of $G$~\cite{Schur}. Let $K$ be a Schur cover of $G$, with $K/Z\cong G$. A maximal clique has the form $A=B/Z$, where $B$ is a maximal abelian subgroup of $K$ (containing $Z$). So $A$ is an abelian subgroup of $G$, and if $pq$ divides $|A|$ (with $p$ and $q$ distinct primes), then $A$ contains an element of order $pq$. Conversely, suppose that $p$ and $q$ are joined in the Gruenberg--Kegel graph, and let $x$ and $y$ be commuting elements of orders $p$ and $q$ in $G$, and $a$ and $b$ their lifts in $K$. Then $a$ and $b$ are contained in $\langle Z,ab\rangle$, which is an extension of a central subgroup by a cyclic group and hence is abelian; so $a$ and $b$ commute. Choosing a maximal abelian subgroup of $K$ containing $a$ and $b$ and projecting onto $G$ gives a maximal clique in the deep commuting graph of $G$, with order divisible by $pq$.

For the power graph, the assertion that an element of order $pq$ is contained
in a clique of size $pq$ fails: for example, the power graph of $C_6$ is not
a clique.
Instead, we use the fact, shown in \cite{ZBM}, that groups with isomorphic power graphs also have isomorphic enhanced power graphs; so they have equal Gruenberg--Kegel graphs, by what has already been proved.\hfill\qed

\paragraph{Proof of Theorem~\ref{t:nullgk}}
The equivalence of $(a)$ and $(b)$ is \cite[Theorem 28]{AACNS}; we give the proof for completeness. If the group $G$ contains an element $g$ of order $pq$, where $p$ and $q$ are primes, then $\langle g^p,g^q\rangle=\langle g\rangle$, so $g^p$ and $g^q$ are joined in the enhanced power graph, but not in the power graph. Conversely, suppose that there are no edges in the Gruenberg--Kegel graph of $G$. Then every element of $G$ has prime power order. Suppose that $g$ and $h$ are joined in the enhanced power graph of $G$, so that they generate a cyclic group, necessarily of prime power order. This group must be generated by one of $g$ and $h$, say $g$; then $h$ is a power of $g$, so $g$ and $h$ are adjacent in the power graph.

We now turn to the equivalence of (b) and (c).

Finite groups whose Gruenberg--Kegel graphs do not have edges (these groups are known as
\emph{EPPO-groups}) were investigated by many authors (see, for example,
\cite{Higman,Shi_Yang,Suzuki}, also see \cite{KondrKhr1,KondrKhr2} and
\cite[Proposition]{AKondTrianFree2}). Let us just summarize these results and
obtain the explicit list of such groups given in (c). This
result does not depend on the classification of finite simple groups.

\begin{lemma}{\rm ({See {\rm \cite{Bus_Gor}}, see also, for example, {\rm \cite[Lemma~1]{Zin_Maz}}})}\label{FrGroups}

Let $G = FH$ is a Frobenius group with kernel $F$ and complement $H$. Then

$(a)$ $F=F(G)$ is the Fitting subgroup of $G$ and $|H|$ divides $|F| - 1$.

$(b)$ Each subgroup of order $pq$ from $H$, where $p$ and $q$ are {\rm(}not necessary distinct{\rm)} primes, is cyclic. In particular, each Sylow subgroup of $H$ is either cyclic or a {\rm(}generalized{\rm)} quaternion group.

$(c)$ If $|H|$ is even, then $H$ contains a unique involution.

$(d)$ If $H$ is non-solvable, then $H$ contains a subgroup $K=S \times Z$, where
$S\cong SL_2(5)$, $(|S|, |Z|) = 1$, and $|H:K| \in \{1,2\}$.
\end{lemma}

\begin{lemma}{\rm (\cite[Lemma~4]{DoJaLu})}\label{BrChar} Let $G$ be a finite simple group, $F$ is a field of characteristic $p > 0$, $V$ is a absolute irreducible $GF$-module, and $\beta$ is a Brauer character of $V$. If $g \in G$ is an element of prime order distinct from $p$, then
$$
dim C_V (g) = (\beta_{\langle g \rangle}, 1_{\langle g \rangle}) = \frac{1}{|g|}\sum_{x \in {\langle g \rangle}} \beta(x).
$$
\end{lemma}

\begin{lemma}{\rm (\cite[Theorem~1]{Higman})}\label{HigmanThm} Assume that every non-trivial element of a finite solvable group $G$ of composite order is of prime power order. Then $|\pi(G)|\le2$.
\end{lemma}

The following assertion is easy to prove, and can be found, for example, in~\cite[Theorem~1]{Lucido2}.

\begin{lemma}\label{non-solvable} Let $G$ be a finite group with ${t(G) \ge 3}$. Then $G$ is non-solvable.
\end{lemma}

\begin{proof} Let $G$ be a finite solvable group. Assume that $\{p,q,r\}$ is an induced $3$-coclique in $\Gamma(G)$. By the Hall Theorem, $G$ contains a $\{p,q,r\}$-Hall subgroup $H$. Then $H$ is solvable, $|\pi(H)|=3$ and $\Gamma(H)$ is a $3$-coclique, a contradiction to Lemma~\ref{HigmanThm}.
\end{proof}

\begin{lemma}{\rm (See {\rm \cite[Theorem~16]{Suzuki}})}\label{EPPO} Assume that every non-trivial element of a finite simple group $G$ of
composite order is of prime power order. Then $G$ is isomorphic to one of the following groups: $\PSL_2(q)$ for $q \in \{5, 7, 8, 9, 17\}$, $\PSL_3(4)$, $\Sz(q)$ for $q \in \{8,32\}$.
\end{lemma}

\begin{lemma}{\rm (See {\rm \cite[Proposition~3.2]{Stewart}})}\label{PSL2qodd} Let $G$ be a finite group, $H \unlhd G$, and $G/H \cong\PSL_2(q)$,
where $q$ is odd and $q > 5$, and let $C_H(t) = 1$ for some element $t$ of order $3$ from $G\setminus H$. Then $H = 1$.
\end{lemma}

\begin{lemma}{\rm (See {\rm \cite[Theorem~8.2]{Higman}}, {\rm \cite[Proposition~4.2]{Stewart}})}\label{PSL2qeven} Let $G$ be a finite group, $1 \not = H \unlhd G$, and  $G/H \cong\PSL_2(2^n)$, where $n \ge 2$. Assume that $C_H(t) = 1$ for some element $t$ of order $3$ from $G$. Then $H = O_2(G)$ and $H$ is the direct product of minimal normal subgroups of $G$, each of which is of order $2^{2n}$ and as a $G/H$-module is isomorphic to the natural $\GF(2^n)SL_2(2^n)$-module.
\end{lemma}

\begin{lemma}{\rm (See {\rm \cite[Theorem, Remark~1]{Martineau}})}\label{Sz(q)}
Let $G$ be a finite group, $1\not= H \unlhd G$, $G/H \cong \Sz(q)$ for $q \in \{8, 32\}$. Assume that $C_H(t) = 1$ for some element $t$ of order $5$ from $G$. Then $H = O_2(G)$ and $H$ is the direct product of minimal normal subgroups of $G$, each of which is of order $2^{4n}$ and as a $G/O_2(G)$-module is isomorphic to the natural $\GF(2^n)\Sz(2^n)$-module of dimension $4$

\end{lemma}

Let $G$ be a finite group such that $\Gamma(G)$ is a coclique.

The case $|\pi(G)|=1$ is clear. Let $|\pi(G)|=2$. In this case $G$ is solvable. Now by Lemma~\ref{Gruenberg--Kegel theorem}, $G$ is a Frobenius group or $2$-Frobenius group. It is easy to see that for any Frobenius group or $2$-Frobenius group $G$ with $|\pi(G)|=2$, $\Gamma(G)$ is a $2$-coclique. Note that a rather detailed description of solvable EPPO-groups can be found in \cite[section~2]{Suzuki}.

Suppose that $|\pi(G)|\ge 3$. Then by Lemma~\ref{non-solvable}, $G$ is non-solvable. By Lemma~\ref{Gruenberg--Kegel theorem}, either $G/F(G)$ is almost simple or $G$ is a non-solvable Frobenius group. By Lemma~\ref{FrGroups}, in the latter case, it is easy to see that $G/O(G)$ contains an element of order $10$, therefore $\Gamma(G)$ is not a coclique. Thus, $G/F(G)$ is almost simple. Moreover, $\Gamma(G/F(G))$ is a coclique (that is, $G/F(G)$ is an EPPO-group), and $|F(G)|\le 1$ since $\pi(F(G))$ form a clique in $\Gamma(G)$. By Lemma~\ref{EPPO}, $Soc(G/F(G))$ is one of the following groups: $\PSL_2(q)$ for $q \in \{5, 7, 8, 9, 17\}$, $\PSL_3(4)$, $\Sz(q)$ for $q \in \{8,32\}$. Note that $|\pi(\PSL_2(q))|=3$ if $q\in \{5,7,8,9,17\}$ and $|\pi(\PSL_3(4)|=|\pi(\Sz(8))|=|\pi(\Sz(32)|=4$. Using informtion in the \emph{Atlas of Finite Groups}~\cite{Atlas} we conclude that $G/F(G)$ is one of the following EPPO-groups{\rm:} $A_5 \cong\PSL_2(4) \cong\PSL_2(5)$, $A_6\cong\PSL_2(9)$, $\PSL_2(7)$, $\PSL_2(8)$, $\PSL_2(17)$, $M_{10}\cong\PSL_2(9).{2_3}$, $\Sz(q)$ for $q \in \{8,32\}$. This list can be found also in \cite[section~3]{Suzuki}.

Assume that $\Soc(G/F(G))$ is $A_6\cong\PSL_2(9)$, $\PSL_2(7)$, or $\PSL_2(17)$. It is easy to see that in this case if $F(G)\not =1$, then $F(G)=O_p(G)$ for some $p \in \pi(G)$. If $p\not = 3$, then by Lemma~\ref{PSL2qodd}, $F(G)=1$. Now suppose that $F(G)=O_3(G)$. Consider an involution $x \in G$. By Lemma~\ref{Gruenberg--Kegel theorem}, the subgroup $O_3(G)\langle x \rangle$ is a Frobenius group with kernel $O_3(G)$ and complement $\langle x \rangle$. Therefore, by Lemma~\ref{FrGroups}, $|O_3(G)|$ divides $|x|-1=1$. Thus, $F(G)=1$.

If $G/F(G) \cong \PSL_3(4)$, then \cite{Atlas} shows that $G/F(G)$ contains a maximal subgroup isomorphic to $\PSL_2(7)$ and again we conclude that $F(G)=1$ as in the previous paragraph.

Assume that $G/F(G)$ is $\Sz(8)$ or $\Sz(32)$. Again if $F(G)\not =1$, then $F(G)=O_p(G)$ for some $p \in \pi(G)$. If $p\not = 5$, then by Lemma~\ref{Sz(q)}, $F(G)=O_2(G)$ is the direct product of minimal normal subgroups of $G$, each of which is of order $2^{4n}$ and as a $G/O_2(G)$-module is isomorphic to the natural $\GF(2^n)\Sz(2^n)$-module of dimension $4$. For converse, Lemma~\ref{BrChar} and tables of the $2$-modular Brauer characters of groups $\Sz(8)$ and $\Sz(32)$ (see Pages~63 and ~197 in~\cite{AtlasBrCh}, respectuvely) imply that in this case $\Gamma(G)$ is a $4$-coclique. Now suppose that $F(G)=O_5(G)$. As above, considering an involution $x \in G$, we conclude that by Lemma~\ref{Gruenberg--Kegel theorem}, the subgroup $O_5(G)\langle x \rangle$ is a Frobenius group with kernel $O_5(G)$ and complement $\langle x \rangle$. Therefore, by Lemma~\ref{FrGroups}, $|O_5(G)|$ divides $|x|-1=1$. Thus, $F(G)=1$.

Assume that $G/F(G)$ is $A_5 \cong\PSL_2(4)$ or $\PSL_2(8)$.  Again if $F(G)\not =1$, then $F(G)=O_p(G)$ for some $p \in \pi(G)$. If $p\not = 3$, then by Lemma~\ref{PSL2qeven}, we conclude as above that $F(G)=O_2(G)$ is the direct product of minimal normal subgroups of $G$, each of which is of order $2^{2n}$ and as a $G/O_2(G)$-module is isomorphic to the natural $\GF(2^n)SL_2(2^n)$-module. For converse, Lemma~\ref{BrChar} and tables of the $2$-modular Brauer characters of groups $A_5$ and $\PSL_2(8)$ (see Pages~2 and~6 from~\cite{AtlasBrCh}, respectively) imply that in this case $\Gamma(G)$ is a $3$-coclique. As above, we prove that if $F(G)=O_3(G)$, then ${F(G)=1}$.\hfill\qed

\section{Acknowledgements}

This work was supported by the Russian Science Foundation (project no. 19-71-10067).

The authors are thankful to Prof.\ Maria Grechkoseeva, Prof.\ Anatoly Kondrat'ev, Prof.\ Wujie Shi, Prof.\ Andrey Vasil'ev, Prof.\ Viktor Zenkov, and an anonymous reviewer for their helpful comments which improved this text.

Some results on this paper were obtained after discussions of the authors during The 3rd Workshop on ``Algebraic Graph Theory and its Applications'', Mathematical Center in Akademgorodok, Zoom-conference, 2--8 November 2020. The authors would like to thank Prof.\ Elena Konstantinova and Dr.\ Grigory Ryabov for their roles in organising the workshop.


\begin{thebibliography}{999}


\bibitem{AACNS}
G.~Aalipour, S.~Akbari, P.~J.~Cameron, R.~Nikandish and F.~Shaveisi,
On the structure of the power graph and the enhanced power graph of a group,
\textit{Electronic J. Combinatorics} \textbf{24(3)} (2017), P3.16.



\bibitem{Akhlaghi_Khatami_Khosravi_2}
Z.~Akhlaghi, M.~Khatami, B.~Khosravi,
Quasirecognition by prime graph of the simple group ${^2}F_4(q)$,
\textit{Acta Math. Hungar.}, \textbf{122}:4 (2009), 387--397.

\bibitem{Akhlaghi_Khatami_Khosravi}
Z.~Akhlaghi, M.~Khatami, B.~Khosravi,
On the prime graph of $D_n(q)$ where $q \in \{2, 5\}$,
\textit{Quaestiones Mathematicae}, \textbf{36}:4 (2013) , 517--535.


\bibitem{Akhlaghi_Khosravi_Khatami_2010}
Z.~Akhlaghi, B.~Khosravi, M.~Khatami,
Characterization by prime graph of $PGL(2, p^k)$ where $p$ and $k>1$ are odd,
\textit{International Journal of Algebra and Computation}, \textbf{20}:07 (2010), 847--873.

\bibitem{ak}
O. A. Alekseeva and A. S. Kondrat'ev,
On recognizability of some finite simple orthogonal groups by spectrum,
\textit{Proc. Steklov Inst. Math.} \textbf{266} (2009), 10--23.

\bibitem{AKondTrianFree2}
O.~A.~Alekseeva, A.~S.~Kondrat’ev,
Finite groups whose prime graphs do not contain triangles. II.
\textit{Proc. Steklov Inst. Math.} \textbf{296} (2017), 19--30.



\bibitem{Amiri_Asboei_Iranmanesh_Tehranian_1}
S.~S.~Amiri, A.~K.~Asboei, A.~Iranmanesh, A.~Tehranian,
Quasirecogniton by prime graph of $U_3(q)$ where $2<q=p^\alpha<100$,
\textit{Int. J. Group Theory}, \textbf{1}:3 (2012), 51--66.

\bibitem{Amiri_Asboei_Iranmanesh_Tehranian_2}
S.~S.~Amiri, A.~K.~Asboei, A.~Iranmanesh, A.~Tehranian,
Quasirecognition by the prime graph of $L_3(q)$ where $3<q=p^\alpha<100$,
\textit{Bull. Iranian Math. Soc.}, \textbf{39}:2 (2013), 289--305.


\bibitem{Babai_Khosravi_2011}
A.~Babai, B.~Khosravi,
Recognition by prime graph of ${^2}D_{2^m+1}(3)$,
\textit{Siberian Math. J.}, \textbf{52}:5 (2011), 788--795.

\bibitem{Babai_Khosravi_2012}
A.~Babai, B.~Khosravi,
On the composition factors of a group with the same prime graph as $B_n(5)$,
\textit{Czechoslovak Math. J.}, \textbf{62}:2 (2012), 469--486.

\bibitem{Babai_Khosravi}
A.~Babai, B.~Khosravi,
Groups with the same prime graph as the simple group $D_n(5)$,
\textit{Ukr. Math. J.}, \textbf{66}:5 (2014), 666--677.


\bibitem{Babai_Khosravi_2014}
A.~Babai, B.~Khosravi,
Quasirecognition by prime graph of $^2D_n(3^\alpha)$ where $n=4m+1\ge21$ and $\alpha$ is odd,
\textit{Math. Notes}, \textbf{95}:3 (2014), 293--303.

\bibitem{Babai_Khosravi_2015}
A.~Babai, B.~Khosravi,
Quasirecognition by prime graph of $L_n(2^\alpha)$ for some $\alpha$ and $n$,
\textit{Math. Reports}, \textbf{17(67)}:1 (2015), 119--132.


\bibitem{Beynekalae_Iranmanesh_Ghasemabadi}
S.~S.~Beynekalae, A.~Iranmanesh, M.~F.~Ghasemabadi,
Quasirecognition by prime graph of the simple group $B_n(2)$,
\textit{Southeast Asian Bull. Math.}, \textbf{39}:2 (2015), 181--193.

\bibitem{BF}
R.~Brauer and K.~A.~Fowler,
On groups of even order,
\textit{Ann. Math.} \textbf{62} (1955), 565--583.

\bibitem{Burness_Covato}
T.~C.~Burness and E.~Covato,
On the prime graph of simple groups,
\textit{Bull. Australian Math. Soc.} \textbf{91}:2 (1995), 227--240.

\bibitem{Bus_Gor}
V.~M.~Busarkin and Yu.~M.~Gorchakov,
\textit{Finite splittable groups}, Nauka, Moscow, 1968 [in Russian].

\bibitem{CK}
P.~J. Cameron and B.~Kuzma,
Between the enhanced power graph and the commuting graph,
\url{https://arxiv.org/2012.03789}.

\bibitem{Chen}
G.~Chen,
A new characterization of sporadic simple groups,
\textit{Algebra Colloq.}, \textbf{3}:1 (1996), 49--58.

\bibitem{Atlas}
J.~H.~Conway \textit{et~al.},
\textit{Atlas of finite groups},
Clarendon Press, Oxford, 1985.

\bibitem{DoJaLu}
S.~Dolfi, E.~Jabara, M.~S.~Lucido, C55-groups,
\textit{Siberian Math. J.} \textbf{45}:6 (2004), 1053--1062.


\bibitem{GKMK}
A.~L.~Gavrilyuk, I.~V.~Khramtsov, A.~S.~Kondrat'ev, N.~V.~Maslova,
On realizability of a graph as the prime graph of a finite group,
\textit{Sibebian Electron. Math. Rep.}, \textbf{11} (2014), 246--257.

\bibitem{Gerono}
G.~C.~Gerono,
Note sur la resolution en nombres entiers et positifs de l’equation $x^m = y^n - 1$,
\textit{Nouv. Ann. Math. {\rm(2)}}, \textbf{9} (1870), 469--471.

\bibitem{Ghasemabadi_Ahanjideh}
 M.~F.~Ghasemabadi, N.~Ahanjideh,
Characterizations of the simple group $D_n(3)$ by prime graph and spectrum,
\textit{Iran. J. Math. Sci. Inform.}, \textbf{7}:1 (2012), 91--106.

\bibitem{Ghasemabadi_Iranmanesh}
M.~F.~Ghasemabadi, A.~Iranmanesh,
Quasirecognition by the prime graph of the group $C_n(2)$, where $n \not= 3$ is odd,
\textit{Bull. Malays. Math. Sci. Soc.}, \textbf{34}:3 (2011), 529--540.

\bibitem{Ghasemabadi_Iranmanesh_2}
M.~F.~Ghasemabadi, A.~Iranmanesh,
2-quasirecognizability of the simple groups $B_n(p)$ and $C_n(p)$ by prime graph,
\textit{Bull. Iran. Math. Soc.}, \textbf{38}:3 (2012), 647--668.

\bibitem{Ghasemabadi_Iranmanesh_Ahanjideh_2012}
M.~F.~Ghasemabadi, A.~Iranmanesh, N.~Ahanjideh,
Characterizations of the simple group ${^2}D_n(3)$ by prime graph and spectrum,
\textit{Monatsh. Math.}, \textbf{168}:3-4 (2012), 347--361.

\bibitem{Ghasemabadi_Iranmanesh_Ahanjideh}
M.~F.~Ghasemabadi, A.~Iranmanesh, N.~Ahanjideh,
2-recognizability of the simple groups $B_n(3)$ and $C_n(3)$ by prime graph,
\textit{Bull. Iranian Math. Soc.}, \textbf{39}:6 (2013), 1273--1281.

\bibitem{Gorenstein_LS}
D.~Gorenstein, R.~Lyons, R.~Solomon,
The classication of the finite simple groups,
\textit{Providence, RI{\rm:} Amer. Math. Soc.}, \textbf{3} (1998).

\bibitem{Gorshkov}
I.~B.~Gorshkov,
Characterization of groups with non-simple socle,
arXiv:2011.15088 [math.GR].

\bibitem{Gorsh_Grish}
I.~B.~Gorshkov, A.~N.~Grishkov,
On recognition by spectrum of symmetric groups,
\textit{Siberian Electron. Math. Rep.}, \textbf{13} (2016), 111--121.

\bibitem{Gors_Mas2}
I.~B.~Gorshkov, N.~V.~Maslova,
Finite almost simple groups whose Gruenberg--Kegel graphs coincide with Gruenberg--Kegel graphs of solvable groups,
\textit{Algebra and Logic}, \textbf{57}:2 (2018), 115--129.

\bibitem{Gorsh_Mas}
I.~B.~Gorshkov, N.~V.~Maslova,
The group $J_4\times J_4$ is recognizable by spectrum,
\textit{J. Algebra and Its Applications}, \textbf{20}:4, 2150061.

\bibitem{Gorsh_Staroletov}
I.~Gorshkov, A.~Staroletov ,
On groups having the prime graph as alternating and symmetric groups,
\textit{Commun. Algebr.}, \textbf{47}:9 (2019), 3905--3914.

\bibitem{Grechkoseeva}
M.~A.~Grechkoseeva,
On spectra of almost simple extensions of even-dimensional orthogonal groups,
\textit{Siberian Math. J.}, \textbf{59}:4 (2018), 623--640.

\bibitem{Grech_Skres}
M.~A.~Grechkoseeva, S.~V.~Skresanov,
On element orders in covers of $L_4(q)$ and $U_4(q)$,
\textit{Siberian Electron. Math. Rep.} \textbf{17} (2020), 585--589.

\bibitem{Grech_Vas}
M.~A.~Grechkoseeva, A.~V.~Vasil’ev,
On the structure of finite groups isospectral to finite simple groups,
\textit{J. Group Theory}, \textbf{18}:5 (2015), 741--759.

\bibitem{GrechVasZvezd}
M.~A.~Grechkoseeva, A.~V.~Vasil'ev, M.~A.~Zvezdina,
Recognition of symplectic and orthogonal groups of small dimensions by spectrum,
\textit{J. Algebra. Appl.}, \textbf{18}:12 (2019), 1950230.

\bibitem{Grech_Zvezd}
M.~A.~Grechkoseeva, M.~A.~Zvezdina,
On recognizability by spectrum of the groups $L_4 (q)$ and $U_4 (q)$,
\textit{Siberian Math. J.}, \textbf{61}:6 (2020), 1039--1065.

\bibitem{Gruber_etal}
A.~Gruber, T.~M.~Keller, M.~L.~Lewis, K.~Naughton, B.~Strasser,
A characterization of the prime graphs of solvable groups,
\textit{J. Algebra}, \textbf{442}:SI (2015), 397--422.

\bibitem{Gruen_Keg}
K.~W.~Gruenberg, O.~Kegel, Unpublished manuscript, 1975.

\bibitem{Gruen_Rogg}
K.~W.~Gruenberg, K.~W.~Roggenkamp,
Decomposition of the augmentation ideal and of the relation modules of a finite group,
\textit{Proc. London Math. Soc. {\rm (3)}}, \textbf{31}:2 (1975), 149--166.

\bibitem{GuoKondrat'evMaslova}
W.~Guo, A.~S.~Kondrat'ev, N.~V.~Maslova,
Recognition of the group $E_6(2)$ by Gruenberg--Kegel graph,
\textit{Trudy Inst. Mat. i Mekh. UrO RAN}, \textbf{27}:4, to appear.


\bibitem{Guralnick_Tiep}
R.~M.~Guralnick, P.~H.~Tiep,
Finite simple unisingular groups of Lie type,
\textit{J. Group Theory}, \textbf{6}:3 (2003), 271--310.

\bibitem{Hagie_2003}
M. Hagie,
The prime graph of a sporadic simple group,
\textit{Comm. Algebra}, \textbf{31}:9 (2003), 4405--4424.

\bibitem{Higman}
G.~Higman,
\textit{Odd characterizations of finite simple groups: lecture notes}.
Michigan: University of Michigan, 1968. 77 pp.

\bibitem{Iiyori_Yamaki_1}
N. Iiyori, H. Yamaki,
Prime graph components of the simple groups of Lie type over the fields of even characteristic,
\textit{Proc. Japan Acad. Ser. A. Math. Sci.}, \textbf{67}:3 (1991), 82--83.

\bibitem{Iiyori_Yamaki_2}
N. Iiyori, H. Yamaki, Prime graph components of the simple groups of Lie type over the fields of even characteristic,
\textit{J. Algebra}, \textbf{155}:2 (1993), 335--343; Corrigenda, J. Algebra, 181:2 (1996), 659.

\bibitem{AtlasBrCh}
C.~Jansen \textit{et~al.},
\textit{An atlas of Brauer characters},
Clarendon Press, Oxford, 1995.

\bibitem{KQ}
A.~V.~Kelarev and S.~J.~Quinn,
A combinatorial property and power graphs of groups,
(Proc. of the Vienna Conference, Vienna, 1999),
\textit{Contrib. General Algebra} \textbf{12} (2000), 229--235.

\bibitem{ABKhosravi}
A.~Khosravi, B.~Khosravi,
$2$-recognizability by prime graph of $\PSL(2, p^2)$,
\textit{Siberian Math. J.}, \textbf{49}:4 (2008), 749--757.

\bibitem{BKhosravi}
B.~Khosravi,
$n$-recognition by prime graph of the simple group $\PSL_2(q)$,
\textit{J. Algebra Appl.}, \textbf{7}:6 (2008), 735--748.

\bibitem{BKhosravi_2}
B.~Khosravi,
Quasirecognition by prime graph of $L_{10}(2)$,
\textit{Siberian Math. J.}, \textbf{50}:2 (2009), 355--359.


\bibitem{BKhosravi_survey}
B.~Khosravi,
On the prime graph of a finite group,
\textit{London Mathematical Society Lecture Note Series}, \textbf{388} (2009), 424--428.


\bibitem{Khosravi_Akhlaghi_Khatami}
B.~Khosravi, Z.~Akhlaghi, M.~Khatami,
Quasirecognition by prime graph of simple group $D_n(3)$,
\textit{Publ. Math. Debrecen}, \textbf{78}:2 (2011), 469--484.


\bibitem{Khosravi_Babai_2011}
B.~Khosravi, A.~Babai,
Quasirecogniton by prime graph of $F_4(q)$ where $q=2^n>2$,
\textit{Monatsh. Math.}, \textbf{162}:3 (2011), 289--296.

\bibitem{3Khosravi_2007_2}
B.~Khosravi, B.~Khosravi, B.~Khosravi,
On the prime graph of $\PSL(2,p)$ where $p>3$ is a prime number,
\textit{Acta Math. Hungar.}, \textbf{116}:4 (2007), 295--307.

\bibitem{3Khosravi_2007}
B.~Khosravi, B.~Khosravi, B.~Khosravi,
Groups with the same prime graph as a CIT simple group,
\textit{Houston J. Math.}, \textbf{33}:4 (2007), 967--977 (electronic).

\bibitem{3Khosravi}
B.~Khosravi, B.~Khosravi, B.~Khosravi,
A characterization of the finite simple group $L_{16}(2)$ by its prime graph,
\textit{Manuscripta Math.}, \textbf{126}:1 (2008), 49--58.

\bibitem{Khosravi_Khosravi_Oskouei}
B.~Khosravi, B.~Khosravi, H.~R.~D.~Oskouei,
On recognition by prime graph of the projective special linear group over $GF(3)$,
\textit{Publications de l Institut Mathematique}, \textbf{95} (109) (2014), 255--266.


\bibitem{Khosravi_Moghanjoghi}
B.~Khosravi, A.~Z.~Moghanjoghi,
Quasirecognition by prime graph of some alternating groups,
\textit{Int. J. Contemp. Math. Sci.}, \textbf{2}:28 (2007), 1351--1358.


\bibitem{Khosravi_Moradi_1}
B.~Khosravi, H.~Moradi,
Quasirecognition by prime graph of finite simple groups $L_n(2)$ and $U_n(2)$,
\textit{Acta. Math. Hung.}, \textbf{132} (2011), 140--153.

\bibitem{Khosravi_Moradi_2}
B.~Khosravi, H.~Moradi,
Quasirecognition by prime graph of some orthogonal groups over the binary field,
\textit{J. Algebra Appl.}, \textbf{11}:3 (2012), 1250056.

\bibitem{Khosravi_Moradi_3}
B.~Khosravi, H.~Moradi,
Quasirecognition by prime graph of finite simple groups ${^2}D_n(3)$,
\textit{Int. J. Group Theory}, \textbf{3}:4 (2014), 47--56.



\bibitem{Kondr1}
A.~S.~Kondrat'ev,
Prime graph components of finite simple groups,
\textit{Math. USSR Sb.} \textbf{67} (1990), 235--247.

\bibitem{Kondr2}
A.~S.~Kondrat'ev,
Gruenberg--Kegel graph of a finite group and its applicatuons,
\textit{Algebra and Linear Optimization, Proc. of Internation Conf. in honor of 90th Birthday S.~N.~Chernikov}, UB RAS: Ekaterinburg, 2002, 141--158 (in Russian).

\bibitem{Kondrat'evE7(2)andE7(3)}
A. S. Kondrat'ev,
Recognizability of groups $E_7(2)$ and $E_7(3)$ by prime graph,
\textit{Proc. Steklov Inst. Math. (Suppl.)}, \textbf{289}:suppl.~1 (2015), 139--145.

\bibitem{Kondrat'ev2E6(2)}
A.~S.~Kondrat'ev,
Recognizability by prime graph of the group $^2E_6(2)$,
\textit{Fundam. Prikl. Mat.}, \textbf{22}:5 (2019), 115--120 (in Russian).

\bibitem{KondrRecSpor1}
A. S. Kondrat'ev,
On the recognizability of sporadic simple groups $Ru$, $HN$, $Fi_22$, $He$, $McL$, and $Co_3$ by the Gruenberg--Kegel graph,
\textit{Trudy Inst. Mat. i Mekh. UrO RAN}, \textbf{25}:4 (2019), 79--87 (in Russian).

\bibitem{KondrRecSpor2}
A. S. Kondrat'ev,
On the recognition of the sporadic simple groups $HS$, $J_3$, $Suz$, $O'N$, $Ly$, $Th$, $Fi_{23}$, and $Fi_{24}$ by the Gruenberg--Kegel graph,
\textit{Siberian Math. J.}, \textbf{61}:6 (2020), 1087--1092.



\bibitem{KondrKhr1}
A.~S.~Kondrat'ev and I.~V. Khramtsov,
On finite triprimary groups,
\textit{Trudy Inst. Mat. i Mekh. UrO RAN} \textbf{16}:3 (2010), 150--158
(in Russian).

\bibitem{KondrKhr2}
A.~S.~Kondrat'ev and I.~V.~Khramtsov,
On finite tetraprimary groups,
\textit{Proc. Steklov Inst. Math.} {\bf 279}:suppl.~1 (2012), 43--61,
\url{https://doi.org/10.1134/S0081543812090040}.

\bibitem{Kondr4}
A.~S.~Kondrat'ev and V.~D.~Mazurov,
Recognizibility of alternating groups of prime degree by their element orders,
\textit{Siberian Math. J.} \textbf{41}:2 (2000), 359--369.

\bibitem{Kourovka}
The Kourovka Notebook: Unsolved Problems in Group Theory. Editors: E. I. Khukhro and V. D. Mazurov.
Sobolev Institute of Mathematics SB RAS, Novosibirsk, 2018.

\bibitem{Lucido2}
M.~C.~Lucido,
The diameter of the prime graph of finite groups,
\textit{J. Group Theory} \textbf{2}:2 (1999), 157--172.

\bibitem{Lytkin}
Y.~V.~Lytkin,
On groups critical with respect to a set of natural numbers,
\textit{Siberian Electron. Math. Rep.}, \textbf{10} (2013), 666--675.

\bibitem{Mahmoudifar}
A.~Mahmoudifar,
Recognition by prime graph of the almost simple group  $PGL (2; 25)$,
\textit{J. Linear and Topological Algebra}, \textbf{5}:1 (2016), 63--66.

\bibitem{Mahmoudifar_Khosravi}
A.~Mahmoudifar, B.~Khosravi,
On quasirecognition by prime graph of the simple groups $A^+_n(p)$ and $A^-_n(p)$,
\textit{J. Algebra Appl.}, \textbf{14}:1 (2015), 1550006.

\bibitem{Martineau}
R.~P.~Martineau,
On 2-modular representations of the Suzuki groups,
\textit{Amer. J. Math.} \textbf{94} (1972), 55--72.

\bibitem{Maslova}
N.~V.~Maslova,
On the coincidence of Gruenberg--Kegel graphs of a finite simple group and its proper subgroup,
\textit{Proc. Steklov Inst. Math.} {\rm(}Suppl.{\rm)}, \textbf{288}:suppl.~1 (2015), 129--141.

\bibitem{Maslova_Conference}
N.~V.~Maslova, I.~N.~Belousov, N.~A.~Minigulov,
Open questions formulated at the 13th School-Conference on Group Theory Dedicated to V. A. Belonogov's 85th Birthday,
\textit{Trudy Inst. Mat. i Mekh. UrO RAN}, \textbf{26}:3 (2020), 275--285 (in Russian).


\bibitem{Maslova_Pagon}
N.~V.~Maslova, D.~Pagon,
On the realizability of a graph as the Gruenberg–Kegel graph of a finite group,
\textit{Siberian Electron. Mat. Rep.}, \textbf{13} (2016), 89--100.


\bibitem{Mazurov_1994_1}
V.~D.~Mazurov,
The set of orders of elements in a finite group,
\textit{Algebra Logic}, \textbf{33} (1994) 49--55.

\bibitem{Mazurov_1997_1}
V. D. Mazurov,
Characterizations of finite groups by sets of orders of their elements,
\textit{Algebra and Logic}, \textbf{36}:1 (1997), 23--32.

\bibitem{Mazurov_1997}
V.~D.~Mazurov,
A characterizations of finite nonsimple groups by the set of orders of their elements,
\textit{Algebra and Logic}, \textbf{36}:3 (1997), 182--192.

\bibitem{Mazurov2004}
V.~D.~Mazurov,
Recognition of finite simple groups $S_4(q)$ by their element orders,
\textit{Algebra and Logic}, \textbf{41}:2 (2002), 93--110.

\bibitem{Mazurov_Shi}
V.~D.~Mazurov, W.~J.~Shi,
A criterion of unrecognizability by spectrum for finite groups,
\textit{Algebra Logic}, \textbf{51} (2012) 160--162.

\bibitem{Moghaddamfar_Shi}
A.~R.~Moghaddamfar, W.~J.~Shi,
The characterization of almost simple groups $PGL(2, p)$ by their element orders,
\textit{Communications in Algebra}, \textbf{32}:9 (2004), 3327--3338.

\bibitem{Momen_Khosravi}
Z.~Momen, B.~Khosravi,
On r-recognition by prime graph of $B_p(3)$ where $p$ is an odd prime,
\textit{Monatsh. Math.}, \textbf{166}:2 (2012), 239--253.

\bibitem{Momen_Khosravi_2}
Z.~Momen, B.~Khosravi,
Groups with the same prime graph as the orthogonal group $B_n(3)$,
\textit{Siberian Math. J.} \textbf{54}:3 (2013), 487--500.

\bibitem{Momen_Khosravi_3}
Z.~Momen, B.~Khosravi,
Quasirecognition by prime graph of the simple group $B_n(9)$,
\textit{Proc. Romanian Acad., Ser.~A}, \textbf{16}:3 (2015), 397--404.



\bibitem{Moradi_Darafsheh_Iranmanesh}
H.~Moradi, M.~R.~Darafsheh, A.~Iranmanesh,
Quasirecognition by prime graph of the groups ${^2}D_{2n}(q)$ where $q<10^5$,
\textit{Mathematics}, \textbf{6}:4 (2018), Paper no.~57, 6 p.

\bibitem{Mosavi_Ahanjideh}
S.~Mosavi, N.~Ahanjideh,
Quasirecognition by prime graph of $C_n(4)$, where $n\ge 17$ is odd,
\textit{Bol. Soc. Parana. Mat. $(3)$}, \textbf{33}:1 (2015), 59--67.

\bibitem{Nosratpour_Darafsheh}
P.~Nosratpour, M.~R.~Darafsheh,
Characterization of the group $G_2(5)$ by the prime graph,
\textit{Ukr. Math. J.}, \textbf{68}:8 (2017), 1308--1313.

\bibitem{Schur}
I.~Schur,
\"Uber die Darstellung der endlichen Gruppen durch
gebrochene lineare Substitutionen,
\textit{J. Reine Angew. Math.} \textbf{127} (1904), 20--50.

\bibitem{Shi}
W.~Shi,
A characteristic property of $A_5$,
\textit{Journal of Southwest-China Teachers' University}, \textbf{2} (1986), 11--14 (in Chinese).

\bibitem{Shi_Conjecture}
W.~Shi,
A new characterization of the sporadic simple groups, in: Group theory,
Proc. 1987 Singapore Group theory conf., Berlin-New York, Walter de Gruyter,
1989, 531—540.


\bibitem{Shi_Yang}
W.~Shi, W.~Yang,
On finite groups with elements of prime power orders,
\textit{Journal of Yunnan Education College}, \textbf{1} (1986), p.2--10 (in Chinese).
English translation by Dr. Jinbao Li:  arXiv:2003.09445v1 [math.GR].


\bibitem{Staroletov1}
A.~Staroletov,
On almost eecognizability by spectrum of simple classical groups,
\textit{Int. J. Group Theory}, \textbf{6}:4 (2017), 7--33.

\bibitem{Staroletov2}
A.~M.~Staroletov,
On recognition of alternating groups by prime graph,
\textit{Siberian Electron. Math. Rep.}, \textbf{14} (2017), 994--1010.

\bibitem{Stewart}
W.~B.~Stewart,
Groups having strongly self-centralizing 3-centralizers,
\textit{Proc. London Math. Soc.} \textbf{426}:4 (1973), 653--680.

\bibitem{Suzuki}
M.~Suzuki,
On a class of doubly transitive groups,
\textit{Ann. Math.} \textbf{75}:1 (1962), 105–145.

\bibitem{Vasil_2005}
A.~V.~Vasil’ev,
On connection between the structure of a finite group and the properties of its prime graph,
\textit{Siberian Math. J.} \textbf{46} (2005) 396--404.

\bibitem{Vasil_2015}
A.~V.~Vasil'ev,
On finite groups isospectral to simple classical groups,
\textit{J. Algebra}, \textbf{423} (2015), 318--374.

\bibitem{GrechMazVas}
A.~V.~Vasil'ev, M.~A.~Grechkoseeva, V.~D.~Mazurov,
Characterization of the finite simple groups by spectrum and order,
\textit{Algebra and Logic}, \textbf{48}:6 (2009), 385--409.

\bibitem{VasilVdov_2005}
A.~V.~Vasiliev, E.~P.~Vdovin,
An adjacency criterion for the prime graph of a finite simple group,
\textit{Algebr Logic}, \textbf{44} (2005) 381--406.

\bibitem{VasilVdov_2011}
A.~V.~Vasil’ev, E.~P.~Vdovin,
Cocliques of maximal size in the prime graph of a finite simple group,
\textit{Algebra Logic}, \textbf{50} (2011), 291--322.

\bibitem{Williams}
J.~S.~Williams,
Prime graph components of finite groups,
\textit{J. Algebra} \textbf{69} (1981), 487--513.


\bibitem{EAtlas}
R.~Wilson et~al.,
Atlas of finite group representations,
\url{http:brauer.maths.gmul.ac.uk/Atlas/}

\bibitem{ZBM}
S.~Zahirović, I.~Bo\v{s}njak, R~Madar\'sz,
A study of enhanced power graphs of finite groups,
\textit{J. Algebra Appl.} \textbf{19} (2020), 2050062.

\bibitem{Zavarnitsine_2003}
A. V. Zavarnitsine,
Recognition of the simple groups $L_3(q)$ by element orders,
\textit{J. Group Theory}, \textbf{7}:1 (2003), 81--97.


\bibitem{Zavarnitsine_2006}
A. V. Zavarnitsine,
Recognition of finite groups by the prime graph,
\textit{Algebra and Logic}, \textbf{45}:4 (2006), 220--231.

\bibitem{Zavarnitsine_2006_1}
A. V. Zavarnitsine,
Recognition of simple groups $U_3(Q)$ by element orders,
\textit{Algebra and Logic}, \textbf{45}:2 (2006), 106--116.


\bibitem{Zavarnitsine_2010}
A.~V.~Zavarnitsine,
Uniqueness of the prime graph of $L_{16}(2)$,
\textit{Siberian Electron. Math. Rep.}, \textbf{7} (2010), 119--121.

\bibitem{Zavarnitsine_2013}
A.~V.~Zavarnitsine,
Finite groups with a five-component prime graph,
\textit{Siberian Math. J.}, \textbf{54}:1 (2013), 40–-46.

\bibitem{Zavarn_Mazurov}
A.~V.~Zavarnitsine, V.~D.~Mazurov,
Element orders in coverings of symmetric and alternating groups,
\textit{Algebra and Logic}, \textbf{38}:3 (1999), 159--170.


\bibitem{Zhang_Shi_Shen}
Q.~Zhang, W.~Shi, R.~Shen,
Quasirecognition by prime graph of the simple groups $G_2(q)$ and ${^2}B_2(q)$,
\textit{J. Algebra Appl.}, \textbf{10}:2 (2011), 309--317.



\bibitem{Zin_Maz}
M.~R.~Zinov'eva and V.~D.~Mazurov,
On finite groups with disconnected prime graph,
\textit{Proc. Steklov Inst. Math.} (Suppl.), \textbf{283}:suppl.~1 (2013),
139--145.

\bibitem{Zsigmondy}
K.~Zsigmondy,
Zur Theorie der Potenzreste,
\textit{Monatsh. Math. Phys.} \textbf{3} (1892), 265--284.


\end{thebibliography}
\end{document}